\documentclass[leqno,10pt]{amsart}
\usepackage{amsmath,amstext,amssymb,amsopn,amsthm,mathrsfs,color}
\usepackage{enumerate}
\usepackage{graphicx}
\usepackage{multirow}

\allowdisplaybreaks

\DeclareMathOperator*{\essup}{ess\,sup}

\DeclareMathOperator{\domain}{Dom}

\DeclareMathOperator{\id}{Id}

\newtheorem{thm}{Theorem}[section]
\newtheorem*{thm*}{Theorem}

\newtheorem{propo}[thm]{Proposition}
\newtheorem{lem}[thm]{Lemma}
\newtheorem{cor}[thm]{Corollary}

\newtheorem*{rem*}{Remark}
\newtheorem{rem}[thm]{Remark}

\def\a{\alpha}

\def\s{\sigma}

\begin{document}

\author[A. Nowak]{Adam Nowak}
\address{Adam Nowak \newline
Instytut Matematyczny,
Polska Akademia Nauk\newline
\'Sniadeckich 8,
00-956 Warszawa, Poland}
\email{adam.nowak@impan.pl}

\author[K. Stempak]{Krzysztof Stempak} 
\address{Krzysztof Stempak     \newline
      Instytut Matematyki i Informatyki,
      Politechnika Wroc\l{}awska       \newline
      Wyb{.} Wyspia\'nskiego 27,
      50--370 Wroc\l{}aw, Poland}
\email{krzysztof.stempak@pwr.wroc.pl}


\footnotetext{
\indent
The first-named author was supported by the National Science Centre of Poland, project no.\
2013/09/B/ST1/02057.
Research of the second-named author supported by funds of the Institute of Mathematics and 
Computer Science, Technical University of Wroc\l aw, project \#{S30097/I-18}.
}

\title[Hankel potential operators]{Potential operators associated with\\ 
	Hankel and Hankel-Dunkl transforms} 

\begin{abstract}
We study Riesz and Bessel potentials in the settings of Hankel transform, modified Hankel 
transform and Hankel-Dunkl transform. We prove sharp or qualitatively sharp pointwise estimates of 
the corresponding potential kernels. Then we characterize those $1\le p,q \le \infty$, 
for which the potential operators satisfy $L^p-L^q$ estimates.
In case of the Riesz potentials, we also characterize
those $1\le p,q \le \infty$, for which two-weight $L^p-L^q$ estimates,
with power weights involved, hold. As a special case of our results, we obtain a full characterization
of two power-weight $L^p-L^q$ bounds for the classical Riesz potentials in the radial case.
This complements an old result of Rubin and its recent reinvestigations by
De N\'apoli, Drelichman and Dur\'an, and Duoandikoetxea.
\end{abstract}

\maketitle

\section{Introduction} \label{sec:intro}
In their seminal article \cite{MS}, Muckenhoupt and Stein outlined a program of development of 
harmonic analysis in the framework of the modified Hankel transform $H_{\a}$.
This context emerges naturally in connection with radial analysis in Euclidean spaces. 
Indeed, it is well known that the Fourier transform of a radial function in 
$\mathbb{R}^n$, $n \ge 1$, reduces directly to the modified Hankel transform of order $\a=n/2-1$. Moreover,
the radial part of the standard Laplacian in $\mathbb{R}^n$ is the Bessel operator
$\frac{d^2}{dx^2} + \frac{2\a+1}{x}\frac{d}{dx}$, $\a=n/2-1$, which is the natural 
`Laplacian' in harmonic analysis associated with $H_{\a}$. In \cite[Section 16]{MS}, among other results,
an analogue of the celebrated Hardy-Littlewood-Sobolev theorem was stated for fractional integrals
(Riesz potentials) corresponding to $H_{\a}$.

Recent years brought a growing interest in harmonic analysis related to Hankel transforms/\! Bessel operators.
For instance, Betancor, Harboure, Nowak and Viviani \cite{BHNV} delivered a thorough study of mapping
properties of maximal operators, Riesz transforms and Littlewood-Paley-Stein type square functions
in the settings of modified and non-modified Hankel transforms. This, as well as many earlier results
(see the references in \cite{BHNV}), was done in dimension one. Recently, harmonic analysis in the context of
Bessel operators was developed in higher dimensions, see Betancor, Castro and Curbelo \cite{BCC0,BCC},
Betancor, Castro and Nowak \cite{BCN}, and Castro and Szarek \cite{CS1}. More recently, in a similar
spirit Castro and Szarek \cite{CS2} investigated fundamental harmonic analysis operators in a wider
Hankel-Dunkl setting. The latter situation is a special and the most explicit case
of a general framework based on the Dunkl Laplacian and the Dunkl transform, when the underlying
Coxeter group is isomorphic to $\mathbb{Z}_2^n$. For more details on the Dunkl theory we refer to the
survey article \cite{Ro}.

In this paper we study Riesz and Bessel potentials associated with Hankel and Hankel-Dunkl transforms in
dimension one. We prove sharp pointwise estimates of the corresponding Riesz potential kernels 
(Theorems \ref{thm:ker} and \ref{thm:ker_D}) and qualitatively sharp pointwise estimates for the related
Bessel potential kernels (Theorems \ref{thm:ker_bes} and \ref{thm:ker_besD}). 
This enables us to characterize those $1\le p,q \le \infty$, for which the potential operators satisfy
$L^p-L^q$ estimates, see Theorem \ref{thm:LpLq}, Corollary \ref{cor:LpLq_H} and Theorem \ref{thm:LpLq_D} for
the results on the Riesz potentials, and Theorems \ref{thm:LpLq_bes}, \ref{thm:LpLq_besN}
and \ref{thm:LpLq_besD} for the results on the Bessel potentials.
Moreover, in case of the Riesz potentials, we also determine those $1 \le p,q \le \infty$, for which
two power-weight $L^p-L^q$ bounds hold, see Theorems \ref{thm:main}, \ref{thm:main_H} and \ref{thm:main_D}.
All these results are Hankel or Hankel-Dunkl counterparts of a series of recent
sharp results concerning potential operators in several classic settings related to discrete orthogonal
expansions: Hermite function expansions \cite{NoSt4}, Jacobi and Fourier-Bessel expansions \cite{NR},
and Laguerre function and Dunkl-Laguerre expansions \cite{NoSt5}. The frameworks studied in
this paper correspond to continuous orthogonal expansions, and the general approach elaborated in 
the above mentioned articles applies here as well. 
However, the present $L^p-L^q$ results have somewhat different flavor.

An interesting by-product of our results is an alternative proof of a radial analogue of the celebrated
two power-weighted $L^p-L^q$ estimates for the classical Riesz potentials due to Stein and Weiss \cite{SW}.
Such an analogue was obtained by Rubin \cite{R} in the eighties of the last century. 
Rubin's result, being apparently overlooked,
was recently reinvestigated and refined by De N\'apoli, Drelichman and Dur\'an \cite{DDD}, 
and Duoandikoetxea \cite{Duo}. In fact, Corollary \ref{cor:main} below slightly extends the above
mentioned results, see the related comments following the statement.

Crucial aspects of our results are their sharpness and completeness.
The latter means, in particular, that in each of the contexts we treat the full admissible range of the
associated parameter of type, which in the Hankel-Dunkl setting manifests in including an `exotic'
case of negative multiplicity functions.
Some parts of our results were obtained earlier, by various authors, which is always commented in the
relevant places according to our best knowledge. In this connection, we mention again the article of
Muckenhoupt and Stein \cite{MS}, and the works of Gadjiev and Aliev \cite{GA} where 
Riesz and Bessel potentials in the context of the modified Hankel transform were investigated,
Thangavelu and Xu \cite{TX} where Riesz and Bessel potentials for the Dunkl transform were introduced
and studied,
Hassani, Mustapha and Sifi \cite{HMS} where for the Riesz potentials the subject was continued, 
Betancor, Mart\'{i}nez and
Rodr\'{i}guez-Mesa \cite{BMRM} where Riesz potentials for the Hankel transform were considered,
and Ben Salem and Touahri \cite{BST} where Bessel potentials for the Dunkl transform were studied.
We note that there is a very wide variety of papers and results pertaining to potential
operators in numerous settings. For instance, Anker \cite{A} investigated Riesz and Bessel potentials
in the framework of non-compact symmetric spaces and his analysis, like ours, was based on sharp 
pointwise estimates of the corresponding kernels.

The Riesz and Bessel potentials we study, defined as integral operators, are naturally connected
with negative powers of the underlying `Laplacians' defined spectrally. This can easily be seen
in case of the Bessel potentials, but the issue is more delicate for the Riesz potentials, see
Propositions \ref{pro:neg_m}, \ref{pro:neg} and \ref{pro:neg_D}. Consequently, our results can be
used to obtain $L^p-L^q$ bounded extensions of negative powers of Bessel operators and the one-dimensional
Dunkl Laplacian.
Note also that our precise description of the potential kernels enables further research,
including quite natural questions of more general weighted inequalities, weak and restricted weak type
estimates, etc. Finally, we remark that acquaintance with the Dunkl theory in its general form is
not necessary to follow the part of the paper related to the Hankel-Dunkl transform.
In fact, this transform comes as a certain symmetrization of the modified Hankel transform.

The paper is organized as follows. In Section \ref{sec:prel} we introduce the three settings investigated
and state the main results. Section \ref{sec:esti} is devoted to deriving sharp or qualitatively sharp
estimates of the relevant potential kernels. In Section \ref{sec:LpLq} we prove $L^p-L^q$ bounds for
the Riesz and Bessel potential operators. 

\subsection*{Notation}
Throughout the paper we use a standard notation consistent with that from \cite{NR, NoSt4, NoSt5}. 
In particular, $X\lesssim Y$ indicates that $X\leq CY$ with a positive constant $C$
independent of significant quantities. We shall write $X \simeq Y$ when simultaneously
$X \lesssim Y$ and $Y \lesssim X$. Furthermore, 
$X\simeq\simeq  Y\exp(-cZ) $ means that there exist positive constants $C, c_1, c_2$, 
independent of significant quantities, such that
$$
C^{-1}Y\exp(-c_1Z)\le X\le C\,Y\exp(-c_2Z).
$$
In a number of places we will use natural and self-explanatory generalizations
of the ``$\simeq \simeq$'' relation, for instance, in connection with certain integrals
involving exponential factors. In such cases the exact meaning will be clear from the context.
By convention, ``$\simeq \simeq$'' is understood as ``$\simeq$'' whenever  no exponential
factors are involved. 
The symbols ``$\vee$'' and ``$\wedge$'' mean the operations of taking maximum and minimum, respectively.

We treat positive kernels and integrals as expressions valued in the extended half-line $[0,\infty]$.
Similar remark concerns expressions occurring in various estimates, with the natural limiting interpretations
like, for instance, $(0^+)^{\beta} = \infty$ when $\beta < 0$.

For definitions and terminology related to $L^p$ and weak $L^p$ spaces see, for instance, 
\cite[Chapter 1]{G}. Given a non-negative weight $w$, by $L^p(w^pd\mu)$ 
we understand the weighted $L^p$ space with respect to a measure
$\mu$. This means that $f \in L^p(w^p d\mu)$ if and only if $w f \in L^p(d\mu)$. The latter allows us to
abuse slightly the notation by admitting also $p=\infty$. Thus, by convention, 
$L^{\infty}(w^{\infty} d\mu)$ consists of all measurable functions $f$ such that $w f$ is essentially bounded.
Given $1\le p \le \infty$, we denote by $p'$ its conjugate exponent, $1/p + 1/p' =1$.

For an easy distinction between the settings of 
modified and non-modified Hankel transform we use calligraphic letters when denoting objects
within the latter context. For objects in the Hankel-Dunkl setting we use the blackboard bold font,
e.g.\ $\mathbb L_\a$, $\mathbb H_\a$.

\section{Preliminaries and statement of results}\label{sec:prel}

Let $\a > -1$. Define the measure
$$
d\mu_{\a}(x) = x^{2\a+1}\, dx
$$
on $\mathbb{R}_+=(0,\infty)$ and the functions
$$
\phi_{\a}(u) = u^{-\a} J_{\a}(u) \qquad \textrm{and} \qquad \varphi_{\a}(u) = \sqrt{u} J_{\a}(u), \qquad u>0,
$$
where $J_{\a}$ denotes the Bessel function of the first kind and order $\a$. 
The \emph{modified Hankel transform} $H_{\a}$ and the (non-modified) 
\emph{Hankel transform} $\mathcal{H}_{\a}$ are given by
$$
H_{\a}f(x) = \int_0^{\infty} \phi_{\a}(xy) f(y)\, d\mu_{\a}(y) \qquad \textrm{and} \qquad
\mathcal{H}_{\a}f(x) = \int_{0}^{\infty} \varphi_{\a}(xy) f(y)\, dy, \qquad x > 0,
$$
for appropriate functions $f$ on $\mathbb{R}_+$. It is well known that $H_{\a} \circ H_{\a} = \id$
and $\mathcal{H}_{\a} \circ \mathcal{H}_{\a} = \id$ on $C_c^{\infty}(\mathbb{R}_+)$. Moreover,
$$
\|H_{\a}f\|_{L^2(\mathbb{R}_+,d\mu_{\a})} = \|f\|_{L^2(\mathbb{R}_+,d\mu_{\a})} \qquad \textrm{and} \qquad
\|\mathcal{H}_{\a}f\|_{L^2(\mathbb{R}_+,dx)} = \|f\|_{L^2(\mathbb{R}_+,dx)}
$$
for $f \in C_c^{\infty}(\mathbb{R}_+)$.
Thus $H_{\a}$ and $\mathcal{H}_{\a}$ extend uniquely to isometric isomorphisms on 
$L^2(\mathbb{R}_+,d\mu_{\a})$ and $L^2(\mathbb{R}_+,dx)$, respectively. These extensions are denoted by
still the same symbols, even though they do not express via the integral formulas in general. We believe
that this will not lead to a confusion, and the exact meaning of $H_{\a}$ and $\mathcal{H}_{\a}$
will always be clear from the context.

Intimately connected with the Hankel transforms are the \emph{Bessel operators}
$$
L_\a=-\frac{d^2}{dx^2}-\frac{2\a+1}x\frac d{dx} \qquad \textrm{and} \qquad
\mathcal{L}_\a=-\frac{d^2}{dx^2}-\frac{1/4-\a^2}{x^2}.
$$
Considered initially on $C_c^{2}(\mathbb{R}_+)$, they are symmetric and positive in
$L^2(\mathbb{R}_+,d\mu_{\a})$ or in $L^2(\mathbb{R}_+,dx)$, respectively. The standard self-adjoint 
extensions of $L_{\a}$ and $\mathcal{L}_{\a}$ (denoted here by the same symbols) are given in terms
of $H_{\a}$ and $\mathcal{H}_{\a}$, respectively. More precisely, we have (see \cite[Section 4]{BS})
$$
L_{\a}f = H_{\a}(y^2H_{\a}f), \qquad 
\domain L_{\a} = \big\{ f \in L^2(\mathbb{R}_+,d\mu_{\a}) : y^2 H_{\a}f \in L^2(\mathbb{R}_+,d\mu_{\a})\big\},
$$
and similarly in case of $\mathcal{L}_{\a}$ and $\mathcal{H}_{\a}$.
Note that
$$
H_{\a}(L_{\a}f)(x) = x^2 H_{\a}f(x) \qquad \textrm{and} \qquad 
\mathcal{H}_{\a}(\mathcal{L}_{\a}f)(x) = x^2 \mathcal{H}_{\a}f(x), \qquad \textrm{a.a.} \;\; x>0,
$$
for $f \in \domain L_{\a}$ or $f \in \domain \mathcal{L}_{\a}$, respectively. These identities are true for
all $x>0$ when $f \in C_c^{2}(\mathbb{R}_+)$. The latter fact may be easily verified directly, since
$\phi_{\a}$ and $\varphi_{\a}$ express eigenfunctions of $L_{\a}$ and $\mathcal{L}_{\a}$,
\begin{equation} \label{efL}
L_{\a} \phi_{\a}(xy) = y^2 \phi_{\a}(xy) \qquad \textrm{and} \qquad 
\mathcal{L}_{\a} \varphi_{\a}(xy) = y^2\varphi_{\a}(xy), \qquad x,y > 0
\end{equation}
(here $L_{\a}$ and $\mathcal{L}_{\a}$ are the differential operators applied in the $x$ variable) and
$J_{\a}$ admits the well known asymptotics
\begin{equation} \label{asym_J}
J_{\a}(u) \simeq u^{\a}, \quad u \to 0^+, \qquad \textrm{and} \qquad
J_{\a}(u)=\mathcal O(u^{-1/2}), \quad u \to \infty.
\end{equation}

The settings of $L_{\a}$ and $\mathcal{L}_{\a}$ are intertwined by the unitary isomorphism
$$ 
M_{\a+1/2}\colon L^2(\mathbb{R}_+,d\mu_{\a}) \to L^2(\mathbb{R}_+,dx), \qquad \textrm{where} \qquad
M_{\eta}f(x) = x^{\eta} f(x), \qquad x >0;
$$
in particular, for $\a=-1/2$ the two contexts coincide. Thus we have 
$\mathcal{H}_{\a}\circ M_{\a+1/2} = M_{\a+1/2}\circ H_{\a}$ in $L^2(\mathbb{R}_{+},d\mu_{\a})$ and
$\mathcal{L}_{\a}\circ M_{\a+1/2} = M_{\a+1/2}\circ L_{\a}$ on $\domain L_{\a}$, and similarly
for other objects, for instance fractional integrals investigated in this paper.

For $\a > -1$, we also consider the \emph{Hankel-Dunkl transform} (cf.\ \cite{D})
\begin{equation} \label{DH}
\mathbb{H}_{\a} f(x) =\int_{-\infty}^{\infty} \overline{\psi_{\a}(xy)} f(y)\, dw_{\a}(y), 
\qquad x \in \mathbb{R}.
\end{equation}
Here
$$
\psi_{\a}(u) =  \frac{1}{2} \Big[\phi_{\a}(|u|) + i u \phi_{\a+1}(|u|)\Big], \qquad u \in \mathbb{R}
$$
(the value $\psi_{\a}(0)$ is understood in a limiting sense) and $w_{\a}$ is the measure on
$\mathbb{R}$ given by
$$
dw_{\a}(x) = |x|^{2\a+1}\, dx.
$$
The integral in \eqref{DH} converges
for decent $f$, in particular for $f \in C_c^{\infty}(\mathbb{R}\setminus\{0\})$. It is known that
$(\mathbb{H}_{\a} \circ \mathbb{H}_{\a}) f(x) = f(-x)$, $f \in C_c^{\infty}(\mathbb{R}\setminus\{0\})$, 
and $\mathbb{H}_{\a}$ extends to an isometry on $L^2(\mathbb{R},dw_{\a})$ (we denote this extension by the
same symbol). For $\a \ge -1/2$ this follows from the general Dunkl theory, see \cite[Theorem 4.26]{J};
the full range $\a > -1$ is treated in \cite[Proposition 1.3]{NoSt0}. 

The one-dimensional \emph{Dunkl Laplacian}
$$
\mathbb{L}_{\a}f(x) = L_{\a}f(x) + (\a+1/2)\frac{f(x)-f(-x)}{x^2}, \qquad x \in \mathbb{R},
$$
considered initially on $C_c^{2}(\mathbb{R}\setminus \{0\})$, is symmetric and positive in
$L^2(\mathbb{R},dw_{\a})$. A natural self-adjoint extension of $\mathbb{L}_{\a}$ is given by
$$
\mathbb{L}_{\a}f = \mathbb{H}_{\a}^{-1} (y^2\mathbb{H}_{\a}f), \qquad
\domain \mathbb{L}_{\a} = \big\{f \in L^2(\mathbb{R},dw_{\a}) : y^2\mathbb{H}_{\a}f 
	\in L^2(\mathbb{R},dw_{\a})\big\}.
$$
Clearly, $\mathbb{H}_{\a}(\mathbb{L}_{\a}f)(x) = x^2 \mathbb{H}_{\a}f(x)$ for $f\in \domain \mathbb{L}_{\a}$
and a.a.\ $x \in \mathbb{R}$. This identity holds for all $x \in \mathbb{R}$ when 
$f \in C_c^{2}(\mathbb{R}\setminus \{0\})$, as can be verified with the aid of \eqref{asym_J} and the
relation
$$
\mathbb{L}_{\a} \psi_{\a}(xy) = y^2 \psi(xy), \qquad x,y \in \mathbb{R},
$$
where $\mathbb{L}_{\a}$ is the differential-difference operator applied in the $x$ variable.

The setting of $\mathbb{H}_{\a}$ and $\mathbb{L}_{\a}$ was discussed in numerous papers,
see for instance \cite{RV}, where the transform was called the \textit{generalized Hankel transform},
or more recent articles \cite{TX,CS2} and references therein.
The parameter $\a$ represents the so-called \emph{multiplicity function}, which is non-negative if and
only if $\a \ge -1/2$. The value $\a = -1/2$ corresponds to the trivial multiplicity function, and
in this case we recover the classical setting of the Fourier transform and the Euclidean Laplacian on
$\mathbb{R}$. When restricted to even functions, $\mathbb{H}_{\a}$ and $\mathbb{L}_{\a}$ reduce to
$H_{\a}$ and $L_{\a}$, and the Hankel-Dunkl setting coincides with the framework of the modified Hankel
transform.

\subsection{The setting of the modified Hankel transform} \label{ssec:m_Hankel}

The integral kernel $\{W_t^{\a}\}_{t>0}$ of the Hankel semigroup $\{\exp(-tL_\a)\}$ is given by
$$
W_t^{\a}(x,y)=\int_0^\infty e^{-u^2t}\phi_\a(xu)\phi_\a(yu)\,d\mu_\a(u),     \qquad x,y>0.
$$
The last integral can be computed, in fact we have
$$
W_t^{\a}(x,y) = 
\frac{1}{2t} \exp\bigg(-\frac{x^2+y^2}{4t}\bigg) (xy)^{-\a} I_{\a}\Big(\frac{xy}{2t}\Big), \qquad x,y>0,
$$
where $I_\a$ denotes the Bessel function of the second kind of order $\a$. 
The function $I_\a$ is strictly positive on $(0,\infty)$ and
satisfies the  well known asymptotics
\begin{equation} \label{asym_I}
I_{\a}(u) \simeq z^{\a}, \quad u \to 0^+, \qquad \textrm{and} 
\qquad I_{\a}(u) \simeq u^{-1/2}e^u, \quad u \to \infty.
\end{equation}

Given $\s>0$, consider the negative power $(L_\a)^{-\s}$ defined in $ L^2(d\mu_\a)$ by means of the 
spectral theorem. From the form in which the spectral resolution of $L_\a$ is defined in terms of $H_\a$
(see \cite[Section 4]{BS}) it follows that
$$
(L_\a)^{-\s}f=H_\a\big(x^{-2\s}H_\a f\big), \qquad f\in \domain(L_\a)^{-\s},
$$
where
$$
\domain (L_\a)^{-\s}=\big\{f\in L^2(d\mu_\a)\colon x^{-2\s}H_\a f\in  L^2(d\mu_\a)\big\}.
$$
Taking into account the formal identity
$$
(L_\a)^{-\sigma}=\frac1{\Gamma(\sigma)}\int_0^\infty e^{-tL_\a} t^{\sigma-1}\,dt,
$$
it is natural to introduce the potential kernel
$$
K^{\a,\s}(x,y) = \frac{1}{\Gamma(\s)} \int_{0}^{\infty} W_t^{\a}(x,y) t^{\s-1}\, dt,
$$
and to consider the corresponding potential operator
$$
I^{\a,\s}f(x)=\int_{0}^{\infty}K^{\a,\s}(x,y)f(y)\,d\mu_\a(y), \qquad x>0, 
$$
with its natural domain
$\domain I^{\a,\s}$ consisting of those functions $f$ for which the above integral converges $x$-a.e.
We will see in a moment that the integral defining $K^{\a,\s}(x,y)$ converges for all $x\neq y$ when 
$\s< \a+1$, and  diverges for all $x,y>0$  when $\s \ge \a+1$. 
Also $K^{\a,\s}(x,y)>0$ and $K^{\a,\s}(x,y)=K^{\a,\s}(y,x)$ for all  $x,y>0$. 
Moreover, the heat kernel, and so the potential kernel, satisfy the homogeneity properties 
$$
W_t^{\a}(x,y) = r^{2(\a+1)}W_{r^2t}^{\a}(rx,ry), \qquad 
K^{\a,\s}(x,y) = r^{2(\a+1-\s)}K^{\a,\s}(rx,ry),\qquad x,y,r>0, 
$$
that result in the following homogeneity of the potential operator: 
\begin{equation} \label{homo_m}
I^{\a,\s}(f_r) = r^{-2\s}(I^{\a,\s}f)_r,\qquad r>0.
\end{equation}
Here $f_r(x)=f(rx)$, $r>0$, denotes the dilation of a function $f$. 

The exact behavior of $K^{\a,\s}(x,y)$ is described in the following.
\begin{thm} \label{thm:ker}
Let $\a > -1$. If  $0< \s < \a+1$, then
$$
K^{\a,\s}(x,y) \simeq (x+y)^{-2\a-1}	
	\begin{cases}
		|x-y|^{2\s-1}, & \s< 1/2,\\
		\log \frac{2(x+y)}{|x-y|}, & \s=1/2,\\
		(x+y)^{2\s-1}, & \s>1/2,
	\end{cases}
$$
uniformly in $x,y > 0$. If $\s \ge \a +1$, then $K^{\a,\s}(x,y)=\infty$ for all $x,y>0$.
\end{thm}
This result enables us to characterize $L^p-L^q$ boundedness of the potential operator $I^{\a,\s}$.

\begin{thm} \label{thm:LpLq}
Let $\a > -1$ and $0< \s < \a+1$. Assume that $1\le p,q\le \infty$. Then
\begin{itemize}
\item[(i)] $L^p(d\mu_{\a}) \subset \domain I^{\a,\s}$ if and only if $p<\frac{\a+1}{\s}$;
\item[(ii)] $I^{\a,\s}$ is bounded from $L^p(d\mu_{\a})$ to $L^q(d\mu_{\a})$ if and only if
$$
\frac{1}q = \frac{1}p - \frac{\s}{\a+1} \quad \textrm{and} \quad 1 < p < \frac{\a+1}{\s} \quad
	\textrm{and} \quad \a \ge -1/2;
$$
\item[(iii)] $I^{\a,\s}$ is bounded from $L^1(d\mu_{\a})$ to weak $L^q(d\mu_{\a})$ for 
$q=\frac{\a+1}{\a+1-\s}$ if and only if $\a \ge -1/2$.
\end{itemize}
\end{thm}
In the case $\a\ge-1/2$ the sufficiency part of Theorem \ref{thm:LpLq} (ii) was known earlier, 
see \cite[Section 16(j)]{MS} and the comments closing Section \ref{ssec:m_Hankel}. 
Apart from that, the result seems to be new.

We now explain the way in which $I^{\a,\s}$ and $(L_\a)^{-\s}$ are connected. 
Since the issue is delicate, our approach will be slightly pedantic. 
As test functions we shall use the space $ H_\a(C_c^\infty)$ which is dense in $L^2(d\mu_\a)$; 
here $C_c^\infty=C_c^\infty(\mathbb{R}_+)$. 
The inclusion $ H_\a(C_c^\infty)\subset L^2(d\mu_\a)$ and the density  
follow from the fact that $H_\a$ extends to an isometry on $L^2(d\mu_\a)$.
But in fact more can be said about $H_{\a}(C_c^{\infty})$.  
Given $g\in C_c^\infty$, $H_\a g$ is continuous on $\mathbb{R}_+$ and
\begin{equation} \label{decH}
H_{\a}g(x) = \mathcal{O}(1), \quad x \to 0^+, \qquad \textrm{and} \qquad H_{\a}g(x) = \mathcal{O}(x^{-k}),
	\quad x \to \infty,
\end{equation}
for each fixed $k \in \mathbb{N}$. The first of these relations is a simple consequence of \eqref{asym_J}.
The second one can be verified by applying $H_{\a}$ to $(L_{\a})^{k}g$ and
then using the symmetry of $L_{\a}$, \eqref{efL}, and again \eqref{asym_J}.
Hence the inclusion
$H_\a(C_c^\infty)\subset L^p(d\mu_\a)$, $1\le p\le \infty$, follows.  
In particular, by Theorem \ref{thm:LpLq} (i), $H_\a(C_c^\infty)\subset \domain I^{\a,\s}$, $0<\s<\a+1$. 
The question of density of $H_{\a}(C_c^{\infty})$ in $L^p(d\mu_{\a})$ spaces is more involved.
In case $\a \ge -1/2$ and $1 < p < \infty$ such density follows from \cite[Theorem 4.7]{ST}.

The next result shows that $I^{\a,\s}$ and $(L_{\a})^{-\s}$ coincide on $H_{\a}(C_c^{\infty})$.
Combined with the comments above and the $L^p-L^q$ results for $I^{\a,\s}$ proved in this paper,
it can be used to obtain $L^p-L^q$ bounded extensions of negative powers of $L_{\a}$.
\begin{propo} \label{pro:neg_m}
Let $\a>-1$ and $0<\s<\a+1$. For every $f\in H_\a(C_c^\infty)$ we have
$$ 
I^{\a,\s}f(x)=(L_\a)^{-\s}f(x),\qquad  \textrm{a.a.} \;\; x>0.
$$
\end{propo}
\begin{proof}
Let $f=H_\a g$, $g\in C_c^\infty$. It was just explained that $f\in \domain I^{\a,\s}$. 
To check that also $f\in \domain (L_\a)^{-\s}$, that is  $(\cdot)^{-2\s}H_\a f\in L^2(d\mu_\a)$, 
note that $f\in L^1(d\mu_\a)$, hence $g=H_\a f$ and, consequently,
the desired property follows. Moreover,
$$
(L_\a)^{-\s}f(x)=H_\a\big((\cdot)^{-2\s}g\big)(x)
= \int_0^\infty \phi_\a(xy)y^{-2\s}g(y)\,d\mu_\a(y), \qquad \textrm{a.a.}\;\; x>0.
$$

On the other hand, using the definitions of $I^{\a,\s}$ and $K^{\a,\s}$, and then interchanging the order
of integration (this is easily seen to be legitimate) gives
\begin{equation} \label{i3}
\Gamma(\s) I^{\a,\s}f(x) = \int_0^{\infty}\int_0^{\infty} W_t^{\a}(x,y) f(y)\, d\mu_{\a}(y)\, t^{\s-1}\, dt.
\end{equation}
We now focus on the inner integral with a fixed $t>0$. Using the definition of $W_t^{\a}(x,y)$ and then
changing the order of integrals (which is justified with the aid of \eqref{asym_J} and \eqref{decH}) we get
\begin{align*}
\int_0^{\infty} W_t^{\a}(x,y) f(y)\, d\mu_{\a}(y) & = 
\int_0^{\infty} \int_0^{\infty} e^{-u^2 t} \phi_{\a}(xu) \phi_{\a}(yu) f(y) \, d\mu_{\a}(y)\, d\mu_{\a}(u)\\
&= \int_0^{\infty} e^{-u^2 t} \phi_{\a}(xu) H_{\a}f(u)\, d\mu_{\a}(u).
\end{align*}
Coming back to the integral in \eqref{i3}, this gives us
$$
\Gamma(\s) I^{\a,\s}f(x) = \int_0^{\infty}\int_0^{\infty} e^{-u^2 t} t^{\s-1} \phi_{\a}(xu) g(u) \, 
	d\mu_{\a}(u)\, dt.
$$
Since the support of $g$ is separated from $0$ and $\infty$, we can once again interchange the order
of integration and evaluate first the integral in $t$, which is precisely $\Gamma(\s)u^{-2\s}$.
The conclusion follows.
\end{proof}

Recall the following classical result of E. M.\ Stein and G.\ Weiss \cite[Theorem $B^*$]{SW}
concerning two-weight $L^p-L^q$ estimates, with power weights involved, for
the Euclidean fractional integral (Riesz potential) 
$$
I^\s f(x)=\int_{\mathbb R^n}\frac{f(y)\, dy}{|x-y|^{n-2\s}},\qquad x\in \mathbb R^n, \quad 0<\s<n/2.
$$
\begin{thm}[Stein \& Weiss]  \label{thm:classic}
Let $n\ge1$ and $0<\s<n/2$. Let $a,b \in \mathbb{R}$ and assume that $1<p\le q<\infty$. 
If $a<n/p'$, $b<n/q$, $a+b\ge0$, and $\frac1q=\frac1p+\frac{a+b-2\s}{n}$, then 
\begin{equation*} 
\big\||x|^{-b}I^\s f\big\|_{L^q(\mathbb R^n,dx)}\lesssim \big\||x|^{a}f\big\|_{L^p(\mathbb R^n,dx)}
\end{equation*} 
uniformly in $f \in L^p(\mathbb{R}^n,|x|^{ap}dx)$. 
\end{thm}

In this paper we obtain the following sharp analogue of Theorem \ref{thm:classic} in the context of the
modified Hankel transform. 
\begin{thm} \label{thm:main}
Let $\a > -1$ and $0 < \s < \a +1$. Let $a,b \in \mathbb{R}$ and assume that $1\le p,q \le \infty$.
\begin{itemize}
\item[(i)] $L^p(x^{ap}d\mu_{\a}) \subset \domain I^{\a,\s}$ 
if and only if
\begin{equation} \label{cnd17}
2\s - \frac{2\a+2}{p} < a < \frac{2\a+2}{p'} \qquad \textrm{(both $\le$ when $p=1$)}.
\end{equation}
\item[(ii)]
The estimate
\begin{equation*}
\big\|x^{-b}I^{\a,\s} f\big\|_{L^q(d\mu_\a)}\lesssim \big\|x^af\big\|_{L^p(d\mu_\a)}
\end{equation*} 
holds uniformly in $f \in L^p(x^{ap}d\mu_{\a})$ 
if and only if 
the following conditions are satisfied: 
\begin{itemize}
\item[(a)] $p \le q$,
\item[(b)] $\frac{1}{q} = \frac{1}p + \frac{a+b-2\s}{2\a+2}$,
\item[(c)] $a < \frac{2\a+2}{p'}$ \quad ($\le$ when $p = q'= 1$),
\item[(d)] $b < \frac{2\a+2}q$ \quad ($\le$ when $p = q'=1$),
\item[(e)] $\frac{1}q \ge \frac{1}p - 2\s$ \quad ($>$ when $p=1$ or $q=\infty$).
\end{itemize}
\end{itemize}
\end{thm}
Notice that (e) is superfluous when $\s > 1/2$, and in case $\s=1/2$ it is equivalent to 
$(p,q)\neq (1,\infty)$. Moreover, because of (b), condition (e) in Theorem \ref{thm:main} 
may be replaced by
\begin{itemize}
\item[(e')] $a+b \ge (2\a+1)\big(\frac{1}q-\frac{1}p\big)$ \quad ($>$ when $p=1$ or $q=\infty$).
\end{itemize}
Finally, observe that (b) is simply forced by the homogeneity \eqref{homo_m}.

It is well known that for $\a=n/2-1$ the setting of the modified Hankel transform corresponds to the radial
framework on $\mathbb{R}^n$, $n \ge 1$. In particular, if $f$ is a radial function on $\mathbb{R}^n$,
$f(x) = f_0(|x|)$, then $-\Delta f(x) = (L_{\a}f_0)(|x|)$ and $I^{\s}f$ and $I^{\a,\s}f_0(|\cdot|)$
coincide up to a constant factor independent of $f$. Moreover, integration of $f$ in $\mathbb{R}^n$
with respect to Lebesgue measure reduces to integration of $f_0$ against $d\mu_{\a}$. 
These standard facts together with Theorem \ref{thm:main} specified to $\a=n/2-1$
lead to the following sharp variant of Theorem \ref{thm:classic} for radially symmetric functions.
\begin{cor} \label{cor:main}
Let $n \ge 1$ and $0 < \s < n/2$. Let $a,b \in \mathbb{R}$ and assume that $1\le p,q \le \infty$. 
The estimate 
\begin{equation*}
\big\||x|^{-b}I^\s f\big\|_{L^q(\mathbb R^n,dx)}\lesssim \big\||x|^{a}f\big\|_{L^p(\mathbb R^n,dx)}
\end{equation*} 
holds uniformly in all radial functions $f \in L^p(\mathbb{R}^n,|x|^{ap}dx)$ 
if and only if the following conditions hold: 
\begin{itemize}
\item[(a)] $p \le q$,
\item[(b)] $\frac{1}{q} = \frac{1}p + \frac{a+b-2\s}{n}$,
\item[(c)] $a < \frac{n}{p'}$ \quad ($\le$ when $p = q'= 1$),
\item[(d)] $b < \frac{n}q$ \quad ($\le$ when $p = q'=1$),
\item[(e)] $a+b \ge (n-1)\big(\frac{1}q-\frac{1}p\big)$ \quad ($>$ when $p=1$ or $q=\infty$).
\end{itemize}
\end{cor} 
Under the assumption $1<p\le q<\infty$ the sufficiency part of Corollary \ref{cor:main} 
was proved by Rubin \cite[Theorem 3]{R} already in 1983. This remarkable result was overlooked
and rediscovered recently by De N\'apoli, Drelichman and Dur\'an \cite[Theorem 1.2]{DDD};
see also 
\cite[Section 5]{DDD} for an interesting application to weighted imbedding theorems.
Previous partial results in the same direction can be found in \cite{GST,HK,V}, 
see the comments in \cite[Section 1]{DDD}.
As for the necessity part, under the assumption $1\le p\le q<\infty$ this question has recently been
studied by Duoandikoetxea \cite[Theorem 5.1]{Duo};
see also \cite[Remark 4.2]{DDD}. Our result completes the previous efforts by 
including $q=\infty$ and also by proving the necessity of the condition $p\le q$.

Given $\s>0$, we also consider analogues of the classical Bessel potentials, $(I+L_{\alpha})^{-\s}$.
These are well defined spectrally on the whole $L^2(d\mu_{\a})$, 
and can be extended to more general functions by
means of an integral representation. Since the integral kernel of the semigroup generated by
$-(I+L_{\a})$ is $\{\exp(-t) W_t^{\a}\}_{t>0}$, we introduce the potential operator
$$
J^{\a,\s}f(x) = \int_0^{\infty} H^{\a,\s}(x,y) f(y) \, d\mu_{\a}(y), \qquad x > 0,
$$
where 
$$
H^{\a,\s}(x,y) = \frac{1}{\Gamma(\s)} \int_0^{\infty} e^{-t} W_t^{\a}(x,y) t^{\s-1}\, dt.
$$
Clearly, $0 < H^{\a,\s}(x,y) < K^{\a,\s}(x,y)$ for all $x,y>0$. 
This implies, in particular, that $J^{\a,\s}$ inherits positive $L^p-L^q$ mapping properties of $I^{\a,\s}$,
and the same is true for weak type estimates. An analogous remark pertains to Bessel and Riesz potentials
in the two other settings of this paper.

The next result provides
qualitatively sharp description of the behavior of $H^{\a,\s}(x,y)$.
\begin{thm} \label{thm:ker_bes}
Let $\a > -1$ and let $\s > 0$. The following estimates hold uniformly in $x,y >0$.
\begin{itemize}
\item[(i)] If $x+y \le 1$, then
$$
H^{\a,\s}(x,y) \simeq \chi_{\{\s > \a+1\}} + \chi_{\{\s=\a+1\}} \log\frac{1}{x+y} + (x+y)^{-2\a-1}
	\begin{cases}
		|x-y|^{2\s-1}, & \s < 1/2,\\
		\log\frac{2(x+y)}{|x-y|}, & \s=1/2,\\
		(x+y)^{2\s-1}, & \s > 1/2.
	\end{cases}
$$
\item[(ii)] If $x+y > 1$, then
$$
H^{\a,\s}(x,y) \simeq \simeq (x+y)^{-2\a-1} \exp\big(-c|x-y|\big)
	\begin{cases}
		|x-y|^{2\s-1}, & \s < 1/2,\\
		1+\log^+ \frac{1}{|x-y|}, & \s=1/2,\\
		1, & \s> 1/2.
	\end{cases}
$$
\end{itemize}
\end{thm}
Thus, among other things, we see that the kernel $H^{\a,\s}(x,y)$ behaves in an essentially different way
depending on whether $(x,y)$ is close to the origin of $\mathbb{R}^2$ or far from it. Moreover, the local
behavior (i) is exactly the same as that of the Riesz potential kernel associated with Laguerre function
expansions of convolution type, see \cite[Theorem 2.1 (i)]{NoSt5}. Furthermore, for $\s < \a+1$
we have $H^{\a,\s}(x,y)\simeq K^{\a,\s}(x,y)$ when $x$ and $y$ stay bounded.

The description of $H^{\a,\s}(x,y)$ from Theorem \ref{thm:ker_bes}
enables a direct analysis of the potential operator $J^{\a,\s}$.
In particular, it allows us to characterize those $1\le p,q \le \infty$, for which $J^{\a,\s}$ is
$L^p-L^q$ bounded. 
Notice that the statement below implicitly contains the fact that $L^p(d\mu_{\a})\subset \domain J^{\a,\s}$
for all $1\le p \le \infty$. Moreover, specified to $p=2$ it allows one to check that 
${J}^{\a,\s}$ coincides in $L^2(d\mu_{\a})$ with the negative power $(I+{L}_{\a})^{-\s}$ defined spectrally.
\begin{thm} \label{thm:LpLq_bes}
Let $\a > -1$, $\s>0$ and $1\le p,q \le \infty$.
\begin{itemize}
\item[(a)] If $\a \ge -1\slash 2$, then $J^{\a,\s}$ is bounded from $L^p(d\mu_{\a})$ to $L^q(d\mu_{\a})$
if and only if
$$
\frac{1}{p} - \frac{\s}{\a+1} \le \frac{1}q \le \frac{1}{p} \quad \textrm{and} \quad
\bigg(\frac{1}p,\frac1{q}\bigg) \notin
    \bigg\{ \Big(\frac{\s}{\a+1},0\Big),\Big(1,1-\frac{\s}{\a+1}\Big)\bigg\}.
$$
\item[(b)] If $\a < -1\slash 2$, then $J^{\a,\s}$ is bounded from $L^p(d\mu_{\a})$ to $L^q(d\mu_{\a})$
if and only if $p=q$.
\end{itemize}
\end{thm}
The sufficiency part of Theorem \ref{thm:LpLq_bes} (a) was known, see \cite{GA} or 
\cite[pp.\,128--129]{AUB} (with $n=1$ and an appropriate reformulation; see also the comments below). 
Here, however, we give
a direct proof which offers a better insight into the structure of $J^{\a,\s}$.
The necessity part, as well as item (b), is new.

In the case $\a\ge -1\slash 2$, the Riesz and Bessel potential operators considered in this paper 
are consistent with those defined by means of a generalized convolution structure in \cite{MS, GA, AUB}.
More precisely, it is well known (see e.g.\  \cite{MS}) that there is a convolution structure 
$*_\a$ defined on functions on $\mathbb{R}_+$ turning $L^1(\mathbb{R}_+, d\mu_\a)$ into a commutative 
Banach algebra, and such that $H_\a(f*_{\a} g)=H_\a f\cdot H_\a g$ (for $\a=n\slash 2-1$ this structure is
inherited from the Euclidean convolution on $\mathbb R^n$). Then the operators $I^{\a,\s}$ and $J^{\a,\s}$
coincide, up to multiplicative constants,
with the $*_\a$-convolution operators with kernels $K_*^{\a,\s}(x)=x^{2\s-2\a-2}$ and
$H_*^{\a,\s}(x)=\int_0^\infty \exp(-x^2\slash (4t)-t)t^{\s-\a-2}\,dt$, respectively. Moreover,
$K^{\a,\s}(x,y)=c_{\a,\s}T^y_\a K_*^{\a,\s}(x)$ and $H^{\a,\s}(x,y)=C_{\a,\s}T^y_\a H_*^{\a,\s}(x)$, where
$T^y_\a$ denotes the generalized translation operator 
such that $f*_\a g(x)=\int_0^\infty T^y_\a f(x)g(y)\,d\mu_\a(y)$.

\subsection{The setting of the non-modified Hankel transform} \label{ssec:Hankel}
It has already been mentioned that this setting is intertwined with the previous one. Thus some of 
the results obtained in the framework of the modified Hankel transform can be
transmitted to the Hankel transform case.

The Hankel heat kernel in this setting is given by
$$
\mathcal W_t^{\a}(x,y) = \int_0^\infty e^{-u^2t}\varphi_\a(xu)\varphi_\a(yu)\,du 
= (xy)^{\a+1/2}W_t^{\a}(x,y).
$$
The related Riesz potential kernel is
$$
\mathcal K^{\a,\s}(x,y) =\frac{1}{\Gamma(\s)} \int_{0}^{\infty} \mathcal W_t^{\a}(x,y) t^{\s-1}\, dt 
= (xy)^{\a+1/2}K^{\a,\s}(x,y),
$$
and the corresponding Riesz potential operator is
\begin{equation} \label{pot_m}
\mathcal I^{\a,\s}f(x)=
\int_{0}^{\infty}\mathcal K^{\a,\s}(x,y)f(y)\,dy=x^{\a+1/2}I^{\a,\s}\big((\cdot)^{-\a-1/2}f\big)(x), 
\qquad x>0. 
\end{equation}
In view of the above connection between the Hankel potential kernels, Theorem \ref{thm:ker}
gives sharp estimates for $\mathcal K^{\a,\s}(x,y)$ as well. 
It is therefore clear that the natural domains 
of $\mathcal I^{\a,\s}$ and $I^{\a,\s}$ satisfy 
\begin{equation} \label{dd}
\domain \mathcal I^{\a,\s}=M_{\a+1/2}\big(\domain I^{\a,\s}\big).
\end{equation}
The corresponding homogeneity properties of the heat kernel and the potential kernel read as
$$
\mathcal W_t^{\a}(x,y) = r\mathcal W_{r^2t}^{\a}(rx,ry),\qquad \mathcal K^{\a,\s}(x,y) = 
r^{1-2\s}\mathcal K^{\a,\s}(rx,ry),\qquad x,y,r>0,
$$
and lead to homogeneity of the potential operator,
$$ 
\mathcal I^{\a,\s}(f_r) = r^{-2\s}(\mathcal I^{\a,\s}f)_r,\qquad r>0.
$$ 

We define the negative power $(\mathcal L_\a)^{-\s}$ and its domain $\domain (\mathcal L_\a)^{-\s}$
by replacing $H_\a$ and $L^2(d\mu_\a)$ by $\mathcal H_\a$ and $L^2(dx)$, respectively,
in the corresponding definitions
in the modified Hankel transform setting. 
It is immediate to check that
$(L_\a)^{-\s}$ and $(\mathcal L_\a)^{-\s}$ are also intertwined in the sense that
$\domain(\mathcal L_\a)^{-\s}=M_{\a+1/2}\big(\domain( L_\a)^{-\s}\big)$, and
$$
(\mathcal L_\a)^{-\s}f=M_{\a+1/2}(L_\a)^{-\s}\big(M_{-\a-1/2}f\big), 
\qquad f\in \domain(\mathcal L_\a)^{-\s}.
$$
Moreover, the fact that $H_\a(C_c^\infty)\subset \domain I^{\a,\s}$ together with 
\eqref{dd} show that $\mathcal H_\a(C_c^\infty)\subset \domain \mathcal I^{\a,\s}$ and
$\mathcal H_\a(C_c^\infty)$ is a dense subspace of $L^2(dx)$.
These facts combined with Proposition \ref{pro:neg_m} justify the following.
\begin{propo} \label{pro:neg}
Let $\a>-1$ and $0<\s<1/2$. For every $f\in \mathcal H_\a(C_c^\infty)$ we have
\begin{equation*}
\mathcal I^{\a,\s}f(x)=(\mathcal L_\a)^{-\s}f(x),\qquad \textrm{a.a.}\;\; x >0.
\end{equation*}
\end{propo}

Relation \eqref{pot_m} between $\mathcal{I}^{\a,\s}$ and $I^{\a,\s}$ and Theorem \ref{thm:main}
allow us to obtain in a straightforward manner a characterization of weighted $L^p-L^q$ boundedness
of $\mathcal{I}^{\a,\s}$.
\begin{thm} \label{thm:main_H}
Let $\a > -1$ and $0 < \s < \a +1$. Let $a,b \in \mathbb{R}$ and assume that $1\le p,q \le \infty$.
\begin{itemize}
\item[(i)] $L^p(x^{ap}dx) \subset \domain \mathcal{I}^{\a,\s}$ 
if and only if
$$
2\s-\frac{1}p - \a - \frac{1}2 < a < \frac{1}{p'} + \a + \frac{1}2 \qquad \textrm{(both $\le$ when $p=1$)}.
$$
\item[(ii)]
The estimate
\begin{equation*}
\big\|x^{-b}\mathcal{I}^{\a,\s} f\big\|_{L^q(dx)}\lesssim \big\|x^af\big\|_{L^p(dx)}
\end{equation*}
holds uniformly in $f \in L^p(x^{ap}dx)$ 
if and only if
the following conditions are satisfied: 
\begin{itemize}
\item[(a)] $p \le q$,
\item[(b)] $\frac{1}{q} = \frac{1}p + a+b-2\s$,
\item[(c)] $a < \frac{1}{p'} + \a + \frac{1}2$ \quad ($\le$ when $p = q'= 1$),
\item[(d)] $b < \frac{1}q + \a + \frac{1}2$ \quad ($\le$ when $p = q'=1$),
\item[(e)] $a+b \ge 0$ \quad ($>$ when $p=1$ or $q=\infty$).
\end{itemize}
\end{itemize}
\end{thm}

Let us distinguish the special case of Theorem \ref{thm:main_H} when no weights are involved, i.e.\
$a=b=0$. 
\begin{cor} \label{cor:LpLq_H}
Let $\a > -1$ and $0 < \s < \a+1$. Assume that $1\le p,q \le \infty$.
\begin{itemize}
\item[(i)] If $\a \ge -1/2$, then $L^p(dx) \subset \domain \mathcal{I}^{\a,\s}$ if and only if
$\frac{1}p > 2\s-\a-1/2$ ($\ge$ if $p=1$), and $\mathcal{I}^{\a,\s}$ is bounded from $L^p(dx)$ to $L^q(dx)$
if and only if
$$
\frac{1}q=\frac{1}p - 2\s \quad \textrm{and} \quad p>1 \quad \textrm{and} \quad q < \infty.
$$
\item[(ii)] If $\a < -1/2$, then $L^p(dx) \subset \domain \mathcal{I}^{\a,\s}$ if and only if
$\a+3/2 > \frac{1}p > 2\s-\a-1/2$ (both $\ge$ if $p=1$),
and $\mathcal{I}^{\a,\s}$ is bounded from $L^p(dx)$ to $L^q(dx)$ if and only if
$$
\frac{1}q = \frac{1}p - 2\s \quad \textrm{and} \quad \frac{1}{p'} > -\a-\frac{1}2 \quad \textrm{and} \quad
    \frac{1}q > -\a-\frac{1}2.
$$
\end{itemize}
\end{cor}
Notice that, in view of Corollary \ref{cor:LpLq_H}, $\mathcal{I}^{\a,\s}$ cannot be $L^p-L^q$ bounded
when $\s \ge 1/2$.
The $L^p-L^q$ boundedness from Corollary \ref{cor:LpLq_H} (i) was known earlier, see
\cite[Theorem 1.3]{BMRM} where the argument was based on the estimate
$$
\mathcal W_t^{\a}(x,y)\le C_\a W_t(x-y), \qquad x,y,t > 0,
$$
or rather its Poisson kernel analogue; here and elsewhere $W_t(u)$ is the Gauss-Weierstrass
kernel on $\mathbb R$,
$$
W_t(u)=\frac{1}{\sqrt{4\pi t}}\exp\bigg(-\frac{u^2}{4t}\bigg).
$$
The quoted result contains also the weak type $(1,\frac{1}{1-2\s})$ estimate
for $\mathcal{I}^{\a,\s}$, in case $\a \ge -1/2$ and $\s< 1/2$.
Our present results show that the following slightly stronger statement is true: given $\s< 1/2$,
the operator $\mathcal{I}^{\a,\s}$ is bounded from $L^1(dx)$ to weak $L^{1/(1-2\s)}(dx)$ if and only if
$\a \ge -1/2$. Indeed, for $\a \ge -1/2$ the kernel $\mathcal{K}^{\a,\s}(x,y)$ is controlled by
$\mathcal{K}^{-1/2,\s}(x,y) = K^{-1/2,\s}(x,y)$,
as easily verified with the aid of the asymptotics \eqref{asym_I}. Then $\mathcal{I}^{\a,\s}$ is controlled
by $I^{-1/2,\s}$, and since $d\mu_{-1/2}(x) = dx$, the weak type of $\mathcal{I}^{\a,\s}$ follows
from Theorem \ref{thm:LpLq} (iii) specified to $\a=-1/2$. On the other hand, $\mathcal{I}^{\a,\s}$
is not even defined on $L^1(dx)$ when $\a < -1/2$, see Corollary \ref{cor:LpLq_H} (ii).

For $\s > 0$, consider the Bessel potentials 
$$
\mathcal{J}^{\a,\s}f(x) = \int_0^{\infty} \mathcal{H}^{\a,\s}(x,y) f(y)\, dy, \qquad x > 0,
$$
where
\begin{equation} \label{linkKH}
\mathcal{H}^{\a,\s}(x,y) = \frac{1}{\Gamma(\s)} \int_{0}^{\infty} e^{-t} \mathcal W_t^{\a}(x,y) t^{\s-1} \, dt
= (xy)^{\a+1/2} H^{\a,\s}(x,y).
\end{equation}
Sharp description of the behavior of $\mathcal{H}^{\a,\s}(x,y)$ follows immediately from 
Theorem \ref{thm:ker_bes}.
Then it is not hard to see that for $\a < -1/2$ the inclusion 
$L^p(dx) \subset \domain \mathcal{J}^{\a,\s}$ can hold only if $p > 2/(2\a+3)$.

The next result gives a complete and sharp description of $L^p-L^q$ boundedness of $\mathcal{J}^{\a,\s}$.
It reveals that for $\a \ge -1/2$, $\mathcal{J}^{\a,\s}$ behaves exactly like $J^{-1/2,\s}$, and thus like
the classical Bessel potential $\mathbb{J}^{-1/2,\s}$ (see Theorem \ref{thm:LpLq_besD} below).
On the other hand, for $\a < -1/2$ the $L^p-L^q$ behavior of $\mathcal{J}^{\a,\s}$ is more subtle, and
partially this is caused by the restriction on $p$ mentioned above.
\begin{thm} \label{thm:LpLq_besN}
Let $\a > -1$, $\s>0$ and $1 \le p,q \le \infty$.  
\begin{itemize}
\item[(a)] If $\a \ge -1/2$, then ${\mathcal{J}}^{\a,\s}$ is bounded from 
$L^p(dx)$ to $L^q(dx)$ if and only if
$$
\frac{1}p - 2\s \le \frac{1}q \le \frac{1}p \quad \textrm{and} \quad
	\bigg(\frac{1}p,\frac{1}q\bigg) \notin \big\{(2\s,0), (1,1-2\s)\big\}.
$$
\item[(b)] If $\a< -1/2$ and $p > 2/(2\a+3)$, 
then ${\mathcal{J}}^{\a,\s}$ is bounded from $L^p(dx)$ to $L^q(dx)$ if and only if
$$
\frac{1}p - 2\s \le \frac{1}q \le \frac{1}p
\quad \textrm{and} \quad \frac{1}q > -\a-\frac{1}2.
$$
\end{itemize}
\end{thm}

Notice that Theorem \ref{thm:LpLq_besN} contains implicitly the inclusions
$L^p(dx) \subset \domain \mathcal{J}^{\a,\s}$ for all $1 \le p \le \infty$ in case $\a \ge -1/2$, and
for $p> \frac{2}{2\a+3}$ in case $\a < -1/2$.
Moreover, Theorem \ref{thm:LpLq_besN} specified to $p=2$ allows one to verify that 
$\mathcal{J}^{\a,\s}$ coincides in $L^2(dx)$ with the negative power $(I+\mathcal{L}_{\a})^{-\s}$
defined spectrally.

\begin{rem}
For the two specific values  $\a=\pm1/2$, the theory presented in Section \ref{ssec:Hankel} 
takes a simpler form due to elementary expressions for the Bessel and the modified Bessel functions,
\begin{equation*}
J_{\pm 1/2}(u)=\sqrt{\frac{2}{\pi u}}
\begin{cases}
        \sin u,\\
        \cos u,
    \end{cases}
\qquad
I_{\pm 1/2}(u)=\sqrt{\frac{2}{\pi u}}
\begin{cases}
        \sinh u,\\
        \cosh u.
    \end{cases}
\end{equation*}
We have
$$
\varphi_{\pm1/2}(u)=\sqrt{\frac{2}{\pi}}
\begin{cases}
\sin u,\\
\cos u,
\end{cases} 
$$
and therefore $\mathcal H_{\pm 1/2}$ is the sine  or the cosine transform on $\mathbb{R}_+$, respectively.
Moreover,
$$
\mathcal W_t^{\pm 1/2}(x,y)=W_t(x-y)\mp W_t(x+y).
$$
Consequently, for $0<\s<\pm 1/2+1 $ and $x,y>0$,
$$
\mathcal K^{\pm 1/2,\s}(x,y)=c_\s\big(|x-y|^{2\s-1}\mp|x+y|^{2\s-1}\big), \qquad 
\mathcal I^{\pm 1/2,\s}f(x)=c_{\s}I^{\s}\tilde f_{\mp}(x),
$$
where $I^\s$ denotes the Euclidean potential operator in dimension one, and $\tilde{f}_{\pm}$ are the
even and odd, respectively, extensions of $f$ to $\mathbb{R}$.
\end{rem}

\subsection{The setting of the Hankel-Dunkl transform} \label{ssec:dunkl}

The integral kernel of the Hankel-Dunkl semigroup $\{\exp(-t\mathbb L_\a)\}$, 
$$
\mathbb{W}_t^{\a}(x,y) = \int_{\mathbb R} e^{-u^2t}\psi_\a(xu)\,\overline{\psi_\a(yu)}\,dw_\a(u), 
\qquad x,y \in \mathbb{R},  
$$
is related to the Hankel heat kernels $W_t^{\a}$ and $W_t^{\a+1}$ through the identity
\begin{equation}\label{rel}
\mathbb W_t^{\a}(x,y)=\frac12\Big[W_t^{\a}(|x|,|y|)+xyW_t^{\a+1}(|x|,|y|)\Big],  \qquad x,y\in \mathbb R
\end{equation}
(the values of $W_t^{\a}(x,y)$ at $x=0$ or $y=0$ are understood in a limiting sense).
Consequently,  the associated Riesz potential kernel, 
$$
\mathbb K^{\a,\s}(x,y) = \frac{1}{\Gamma(\s)} \int_{0}^{\infty} \mathbb W_t^{\a}(x,y) t^{\s-1}\, dt,
$$
and the corresponding potential operator,
$$
\mathbb I^{\a,\s}f(x)=\int_{\mathbb R}\mathbb K^{\a,\s}(x,y)f(y)\,dw_\a(y), \qquad x \in \mathbb{R}, 
$$
can be expressed by their modified Hankel transform analogues as
$$
\mathbb K^{\a,\s}(x,y)=\frac12\Big[K^{\a,\s}(|x|,|y|)+xyK^{\a+1,\s}(|x|,|y|)\Big],  \qquad x,y\in \mathbb R,
$$
and 
$$
\mathbb I^{\a,\s}f(x)=I^{\a,\s}f_e(|x|)+xI^{\a+1,\s}((\cdot)^{-1}f_o)(|x|),  \qquad x\in \mathbb R,
$$
respectively. Here $f_{e}$ and $f_{o}$ stand for the even and odd parts of $f$, respectively, restricted
to $\mathbb{R}_+$.

The heat kernel and the potential kernel satisfy the homogeneity properties 
$$
\mathbb{W}_t^{\a}(x,y) = r^{2(\a+1)}\mathbb{W}_{r^2t}^{\a}(rx,ry),
	\qquad \mathbb{K}^{\a,\s}(x,y) = r^{2(\a-\s+1)}\mathbb{K}^{\a,\s}(rx,ry),\qquad x,y,r>0, 
$$
that result in homogeneity of the potential operator, 
$$
\mathbb{I}^{\a,\s}(f_r) = r^{-2\s}(\mathbb{I}^{\a,\s}f)_r,\qquad r>0.
$$
Furthermore, by \eqref{rel} and the asymptotics \eqref{asym_I}  
it is easily seen that $|\mathbb W_t^{\a,\s}(x,y)| \lesssim W_t^{\a,\s}(|x|,|y|)$, $x,y \in \mathbb{R}$, 
and $\mathbb W_t^{\a,\s}(x,y) \simeq W_t^{\a,\s}(x,y)$, $x,y > 0$, hence we have
\begin{align}
|\mathbb K^{\a,\s}(x,y)| & \lesssim K^{\a,\s}(|x|,|y|), \qquad x,y \in \mathbb{R}, \label{kkp1}\\
\mathbb K^{\a,\s}(x,y) & \simeq K^{\a,\s}(x,y), \qquad x,y > 0. \label{kkp2}
\end{align}
Note that $\mathbb H_{-1/2}$ is the usual Fourier transform on $\mathbb R$ and
$\mathbb W_t^{-1/2}(x,y)=W_t(x-y)$ is the Gauss-Weierstrass kernel. Consequently, for $\s < 1/2$
$$
 \mathbb K^{-1/2,\s}(x,y) = c_{\s}|x-y|^{2\s-1}
$$
is the classical Riesz potential kernel.

In another way, the identity \eqref{rel} can be written as
$$
\mathbb{W}_t^{\a}(x,y) = \frac{1}2 (2t)^{-\a-1} 
\exp\bigg(-\frac{x^2+y^2}{4t}\bigg)	\Phi_{\a}\Big(\frac{xy}{2t}\Big),
$$
where $\Phi_{\a}$ is the function on the real line given by
$$
\Phi_{\a}(u) = \frac{I_{\a}(u)}{u^{\a}} + u \frac{I_{\a+1}(u)}{u^{\a+1}},
$$
with proper interpretation of the ratios when $u \le 0$; see e.g.\ \cite{Ro}.
An analysis similar to that performed in \cite[Section 3.2]{NoSt5} shows that 
$\mathbb{W}_t^{\a}(x,y)$ takes both positive and negative values when $\a<-1/2$.
This suggests that, similarly as in the Dunkl-Laguerre setting considered in \cite{NoSt5},
the potential kernel also has this property and hence cannot be sharply estimated in the spirit of
Theorem \ref{thm:ker} when $\a < -1/2$. Thus in the next result we restrict to $\a \ge -1/2$, i.e.\ to
non-negative multiplicity functions.

\begin{thm} \label{thm:ker_D}
Let $\a\ge-1/2$ and let $0 < \s < \a+1$. Then
$$
\mathbb K^{\a,\s}(x,y) \simeq (|x|+|y|)^{-2\a-1}
	\begin{cases}
		|x-y|^{2\s-1}, & \s<1/2,\\
		\log\frac{2(|x|+|y|)}{|x-y|}, & \s=1/2,\\
		(|x|+|y|)^{2\s-1}, & \s>1/2,
	\end{cases}
$$
uniformly in $x,y \in \mathbb{R}$.
\end{thm}

The analogue of Theorem \ref{thm:main} in the present setting is the following.
\begin{thm} \label{thm:main_D}
Let $\a > -1$ and $0 < \s < \a +1$. Let $a,b \in \mathbb{R}$ and assume that $1\le p,q \le \infty$.
\begin{itemize}
\item[(i)] $L^p(|x|^{ap}dw_{\a}) \subset \domain \mathbb{I}^{\a,\s}$ 
if and only if
$$
2\s - \frac{2\a+2}{p} < a < \frac{2\a+2}{p'} \qquad \textrm{(both $\le$ when $p=1$)}.
$$
\item[(ii)]
The estimate
\begin{equation*} 
\big\||x|^{-b}\mathbb{I}^{\a,\s} f\big\|_{L^q(dw_\a)}\lesssim \big\||x|^af\big\|_{L^p(dw_\a)}
\end{equation*} 
holds uniformly in $f \in L^p(|x|^{ap}dw_{\a})$ 
if and only if 
the following conditions are satisfied: 
\begin{itemize}
\item[(a)] $p \le q$,
\item[(b)] $\frac{1}{q} = \frac{1}p + \frac{a+b-2\s}{2\a+2}$,
\item[(c)] $a < \frac{2\a+2}{p'}$ \quad ($\le$ when $p = q'= 1$),
\item[(d)] $b < \frac{2\a+2}q$ \quad ($\le$ when $p = q'=1$),
\item[(e)] $\frac{1}q \ge \frac{1}p - 2\s$ \quad ($>$ when $p=1$ or $q=\infty$).
\end{itemize}
\end{itemize}
\end{thm}

Having in mind Theorem \ref{thm:LpLq}, in the unweighted case we also state the following.
\begin{thm} \label{thm:LpLq_D}
Let $\a > -1$ and $0< \s < \a+1$. Assume that $1\le p,q\le \infty$. Then
\begin{itemize}
\item[(i)] $L^p(dw_{\a}) \subset \domain \mathbb{I}^{\a,\s}$ if and only if $p<\frac{\a+1}{\s}$;
\item[(ii)] $\mathbb{I}^{\a,\s}$ is bounded from $L^p(dw_{\a})$ to $L^q(dw_{\a})$ if and only if
$$
\frac{1}q = \frac{1}p - \frac{\s}{\a+1} \quad \textrm{and} \quad 1 < p < \frac{\a+1}{\s} \quad
	\textrm{and} \quad \a \ge -1/2;
$$
\item[(iii)] $\mathbb{I}^{\a,\s}$ is bounded from $L^1(dw_{\a})$ to weak $L^q(dw_{\a})$ for 
$q=\frac{\a+1}{\a+1-\s}$ if and only if $\a \ge -1/2$.
\end{itemize}
\end{thm}

For $\a \ge -1/2$ items (ii) and (iii) of Theorem \ref{thm:LpLq_D} are essentially contained in
the multi-dimensional results \cite[Proposition 4.2 and Theorem 4.3]{TX}.
Similar results in a general setting of the Dunkl transform and an arbitrary group of reflections
can be found in \cite{HMS}.

The negative power $(\mathbb L_\a)^{-\s}$ and its domain $\domain (\mathbb L_\a)^{-\s}$ is defined
similarly as in the
modified Hankel transform setting, simply by replacing $H_\a$ and $L^2(d\mu_\a)$ by  
$\mathbb H_\a$ and $L^2(dw_\a)$, respectively, in the relevant definitions in Section \ref{ssec:m_Hankel}.
The following analogue of Propositions \ref{pro:neg_m} and \ref{pro:neg} holds.
\begin{propo} \label{pro:neg_D}
Let $\a>-1$ and $0<\s<\a+1$. 
For every $f\in \mathbb H_\a\big(C_c^\infty(\mathbb R\setminus\{0\})\big)$ we have
$$
\mathbb I^{\a,\s}f(x)=(\mathbb L_\a)^{-\s}f(x),\qquad \textrm{a.a.}\;\; x\in \mathbb{R}.
$$
\end{propo}

For $\s>0$, we consider also the Bessel potentials
$$
\mathbb{J}^{\a,\s}f(x) = \int_{\mathbb{R}} \mathbb{H}^{\a,\s}(x,y) f(y)\, dw_{\a}(y), \qquad x \in \mathbb{R},
$$
with
$$
\mathbb{H}^{\a,\s}(x,y)   = \frac{1}{\Gamma(\s)} \int_0^{\infty} e^{-t} 
\mathbb{W}_t^{\a}(x,y) t^{\s-1}\, dt 
 = \frac{1}2 \Big[ H^{\a,\s}(|x|,|y|) + xy H^{\a+1,\s}(|x|,|y|) \Big], \qquad x,y \in \mathbb{R}.
$$
Similarly to \eqref{kkp1} and \eqref{kkp2}, we have
\begin{align}
|\mathbb{H}^{\a,\s}(x,y)| & \lesssim H^{\a,\s}(|x|,|y|), \qquad x,y \in \mathbb{R}, \label{Bkkp1}\\
\mathbb{H}^{\a,\s}(x,y) & \simeq H^{\a,\s}(x,y), \qquad x,y > 0. \label{Bkkp2}
\end{align}
These relations, together with Theorem \ref{thm:LpLq_bes}, 
enable a characterization of $L^p-L^q$ boundedness of $\mathbb{J}^{\a,\s}$. Nevertheless,
it is interesting to find the exact behavior of $\mathbb{H}^{\a,\s}(x,y)$. The next result delivers
qualitatively sharp estimates of this kernel in case $\a > -1/2$ (the case of a positive multiplicity
function). In case $\a < -1/2$ the kernel should be expected to take both positive and negative values.
\begin{thm} \label{thm:ker_besD}
Let $\a>-1/2$ and let $\s > 0$. The following bounds hold uniformly in $x,y \in \mathbb{R}$.
\begin{itemize}
\item[(A)] Assume that $xy \ge 0$, i.e.\ $x$ and $y$ have the same sign.
\begin{itemize}
\item[(Ai)] If $|x|+|y| \le 1$, then
\begin{align*}
\mathbb{H}^{\a,\s}(x,y) & \simeq  \chi_{\{\s>\a+1\}} + \chi_{\{\s=\a+1\}} \log\frac{1}{|x|+|y|} \\ 
& \quad	
 + (|x|+|y|)^{-2\a-1} 
	\begin{cases}
		|x-y|^{2\s-1}, & \s< 1/2,\\
		\log \frac{2(|x|+|y|)}{|x-y|}, & \s=1/2,\\
		(|x|+|y|)^{2\s-1}, & \s>1/2.
	\end{cases}
\end{align*}
\item[(Aii)] If $|x|+|y| > 1$, then
$$
\mathbb{H}^{\a,\s}(x,y) \simeq \simeq (|x|+|y|)^{-2\a-1} \exp\big(-c|x-y|\big) 
	\begin{cases}
		|x-y|^{2\s-1}, & \s< 1/2,\\
		1+\log^+ \frac{1}{|x-y|}, & \s=1/2,\\
		1, & \s>1/2.
	\end{cases}
$$
\end{itemize}
\item[(B)] Assume that $xy< 0$, i.e. $x$ and $y$ have opposite signs.
\begin{itemize}
\item[(Bi)] If $|x|+|y| \le 1$, then
\begin{align*}
\mathbb{H}^{\a,\s}(x,y) & \simeq  \chi_{\{\s>\a+1\}} + \chi_{\{\s=\a+1\}} \log\frac{1}{|x|+|y|} 
	 + (|x|+|y|)^{2\s-2\a-2}. 
\end{align*}
\item[(Bii)] If $|x|+|y| > 1$, then
\begin{align*}
\mathbb{H}^{\a,\s}(x,y) & \simeq \simeq (|x|+|y|)^{-2\a-1} \exp\big(-c|x+y|\big) 
	(|x|+|y|)^{-2}.
\end{align*}
\end{itemize}
\end{itemize}
\end{thm}
For $\a = -1/2$ the kernel $\mathbb{H}^{\a,\s}(x,y)$ corresponds to the classical one-dimensional Bessel
potential. For the sake of completeness, recall that
$$
\mathbb{H}^{-1/2,\s}(x,y) \simeq \simeq \exp\big( -c|x-y|\big)
	\begin{cases}
		|x-y|^{2\s-1}, & \s < 1/2,\\
		1+ \log^+\frac{1}{|x-y|}, & \s=1/2,\\
		1, & \s > 1/2,
	\end{cases}
$$
uniformly in $x,y \in \mathbb{R}$.
Actually, there is an explicit formula for $\mathbb{H}^{-1/2,\s}(x,y)$ involving Macdonald's function,
which leads to more precise asymptotics than the above, even in the multi-dimensional case; see
\cite[p.\,416--417]{AS}. Notice that $\mathbb{H}^{-1/2,\s}(x,y)$ has an exponential decay along the line
$y=-x$, which is not the case of $\mathbb{H}^{\a,\s}(x,y)$ when $\a > -1/2$, cf.\ the comments following
\cite[Theorem 2.4]{NoSt5}.

Finally, we establish a sharp description of $L^p-L^q$ boundedness of $\mathbb{J}^{\a,\s}$,
which happens to coincide with that for $J^{\a,\s}$.
\begin{thm} \label{thm:LpLq_besD}
Let $\a > -1$, $\s>0$ and $1\le p,q \le \infty$.
\begin{itemize}
\item[(a)] If $\a \ge -1\slash 2$, then $\mathbb{J}^{\a,\s}$ is bounded from 
$L^p(dw_{\a})$ to $L^q(dw_{\a})$
if and only if
$$
\frac{1}{p} - \frac{\s}{\a+1} \le \frac{1}q \le \frac{1}{p} \quad \textrm{and} \quad
\bigg(\frac{1}p,\frac1{q}\bigg) \notin
    \bigg\{ \Big(\frac{\s}{\a+1},0\Big),\Big(1,1-\frac{\s}{\a+1}\Big)\bigg\}.
$$
\item[(b)] If $\a < -1\slash 2$, then $\mathbb{J}^{\a,\s}$ is bounded from 
$L^p(dw_{\a})$ to $L^q(dw_{\a})$
if and only if $p=q$.
\end{itemize}
\end{thm}
Note that this result in the classical case $\a=-1/2$ was known earlier, see \cite[p.\,470]{AS}.
The sufficiency part of (a) is partially contained in \cite[Theorems 4.5 and 4.6]{TX}.
Apart from that the theorem is new.
Note also that Theorem \ref{thm:LpLq_besD} specified to $p=2$ allows one to ensure that 
$\mathbb{J}^{\a,\s}$ coincides in $L^2(dw_{\a})$ with the negative power $(I+\mathbb{L}_{\a})^{-\s}$
defined spectrally.

The definitions of the Riesz and Bessel potentials in the Hankel-Dunkl setting considered 
in this paper in the framework of $\mathbb R$ with the reflection group isomorphic to $\mathbb Z_2$
are in the case $\a\ge-1/2$ consistent (up to multiplicative constants)  with those 
investigated in the literature in a general framework of an arbitrary finite reflection group in 
$\mathbb R^d$. See the papers \cite{TX,HMS,BST}.

To explain this, we look only at the Bessel potentials since for the Riesz potentials one can argue similarly
(merely by neglecting the factor $e^{-t}$
in the relevant places). To keep our explanation concise we follow the notation from \cite{BST} 
changing only the character $\a$ to $2\s$ in order to avoid a
notational collision; the reader may also consult the survey \cite{Ro} for
necessary details. 

For a fixed reflection group on $\mathbb R^d$, let $\gamma, \tau_y, *_\gamma, w_\gamma$ 
denote respectively:
an index associated to a multiplicity function, a Dunkl-type generalized translation and 
convolution, and a weight function. The Bessel-Dunkl potential operator is then 
defined as 
$$
\mathcal J^\gamma_{2\s}f=b^\gamma_{2\s}*_\gamma f,
$$
where 
$$
b^\gamma_{2\s}(x)=
	\frac{c_{d,\gamma}}{\Gamma(\s)}\int_0^\infty e^{-t}t^{-\gamma-d/2}\exp\big(-|x|^2/(4t)\big)t^{\s-1}dt,
$$
see \cite[(3.1) and (3.2)]{BST}. 
(Note that evaluating the latter integral with the factor $e^{-t}$ removed results 
in $|x|^{2\s-2\gamma-d}$ times a constant, which is the generalized convolution kernel appearing in the 
definition of the Riesz potential of order $2\s$, see \cite{TX,HMS}.) In fact, see \cite[(3.6)]{BST},
$\mathcal J^\gamma_{2\s}$ is an integral operator with the kernel
$$
\mathcal J^\gamma_{2\s}(x,y)=\frac1{\Gamma(\s)}\int_0^\infty e^{-t}\tau_{-y}\big(F^\gamma_t\big)(x)t^{\s-1}dt,
$$
where $F^\gamma_t(u)=(2t)^{-\gamma-d/2}\exp\big(-|u|^2/(4t)\big)$ is, up to a multiplicative constant,
the \textit{modified Gauss kernel} and
$$
\tau_{-y}\big(F^\gamma_t\big)(x)=c_{\gamma}\int_{\mathbb R^d}
 e^{-t|u|^2}E_\gamma(ix,u)E_\gamma(-iy,u)w_\gamma(u)\,du
$$
is the \textit{Dunkl-type heat kernel}, see \cite[p.\,123]{Ro}; $E_\gamma(\cdot,\cdot)$ 
denotes here the \textit{Dunkl kernel}. 

Coming back to our specific case of $\mathbb R$ where the parameter $\a$ represents the multiplicity
 function, it may be seen that $E_\gamma(ix,u)$ and $E_\gamma(-iy,u)$ appearing above are, up to 
 a multiplicative constant,
equal to $\psi_\a(xu)$ and $\overline{\psi_\a(yu)}$, respectively; 
see \cite[Example 2.1]{Ro} or \cite{NoSt0}. This explains the consistence indicated above.

\section{Estimates of the potential kernels} \label{sec:esti}

In this section we prove Theorems \ref{thm:ker}, \ref{thm:ker_bes}, \ref{thm:ker_D} and \ref{thm:ker_besD}.
We begin with a technical result that provides sharp description of the integral $E_A(T,S)$
defined below. This is essentially \cite[Lemma 2.3]{NoSt4}, see also \cite[Lemma 3.2]{NoSt5}.
Let
\begin{equation*}
E_A(T,S) = \int_0^1 t^A \exp\big(-Tt^{-1}-St\big)\, dt, \qquad 0\le T, S <\infty. 
\end{equation*}
\begin{lem}[{\cite[Lemma 2.3]{NoSt4}}] \label{lem:E}
Let $A \in \mathbb{R}$. Then
$$
E_A(T,S) \simeq \simeq \exp\Big( -c\sqrt{T(T\vee S)}\Big)
\begin{cases}
T^{A+1}, & A< -1,\\
1+\log^+ \frac{1}{T(T\vee S)}, & A = -1,\\
(S\vee 1)^{-A-1}, & A>-1,
\end{cases}
$$
uniformly in $T,S \ge 0$.
\end{lem}

\subsection{Estimates of the Hankel potential kernels}

\begin{proof}[{Proof of Theorem \ref{thm:ker}}]
Using the standard asymptotics \eqref{asym_I} we find that
\begin{equation} \label{W_as}
W_t^{\a}(x,y) \simeq  \begin{cases}
												t^{-\a-1} \exp\left(-\frac{x^2+y^2}{4t}\right),& xy \le t,\\
												(xy)^{-\a-1/2}\frac{1}{\sqrt{t}}\exp\left(-\frac{(x-y)^2}{4t}\right), & xy > t.
											\end{cases}
\end{equation}											
Consequently,
\begin{align*}
K^{\a,\s}(x,y) & \simeq (xy)^{-\a-1/2}\int_0^{xy} t^{\s-3/2}\exp\bigg( -\frac{(x-y)^2}{4t}\bigg)\, dt
	+ \int_{xy}^{\infty} t^{\s-\a-2} \exp\bigg( -\frac{x^2+y^2}{4t}\bigg) \, dt \\
& \equiv I_0 + I_{\infty}.
\end{align*}
Changing the variables of integrations $t \mapsto xyt$ and $t \mapsto xy/t$, respectively, we arrive at
\begin{equation*}
I_0  = (xy)^{\s-\a-1} E_{\s-3/2}\bigg( \frac{(x-y)^2}{4xy}, 0\bigg), \qquad
I_{\infty} = (xy)^{\s-\a-1} E_{\a-\s}\bigg( 0, \frac{x^2+y^2}{4xy}\bigg);
\end{equation*}
notice that, in view of Lemma \ref{lem:E}, $I_{\infty} = \infty$ when $\s\ge \a+1$.
Now applying Lemma \ref{lem:E} twice we get
$$
I_{\infty} + I_{0} \simeq \simeq (x+y)^{2\s-2\a-2} + (xy)^{-\a-1/2}
	\exp\bigg( -c\frac{(x-y)^2}{xy}\bigg) 
	\begin{cases}
		|x-y|^{2\s-1}, & \s < 1/2, \\
		1+\log^{+}\frac{xy}{(x-y)^2}, & \s=1/2, \\
		(xy)^{\s-1/2}, & \s > 1/2.
	\end{cases}
$$

To proceed, we consider two cases. If $(x-y)^2 \le xy$, then $xy \simeq (x+y)^2$. So in this case
$$
K^{\a,\s}(x,y) \simeq (x+y)^{2\s-2\a-2} + (x+y)^{-2\a-1}
	\begin{cases}
		|x-y|^{2\s-1}, & \s< 1/2,\\
		1+\log \frac{x+y}{|x-y|}, & \s=1/2,\\
		(x+y)^{2\s-1}, & \s>1/2.
	\end{cases}
$$
Since the second term on the right-hand side above is the dominating one, the desired bounds follow.

In the opposite case, when $(x-y)^2 > xy$, we observe that $x$ and $y$ are non-comparable in the sense
that either $y< \mathcal{C}^{-1}x$ or $y> \mathcal{C}x$ for a fixed $\mathcal{C}>1$.
For symmetry reasons, we may assume that $y < x$. Then the bounds we must verify take the form
$K^{\a,\s}(x,y) \simeq x^{2\s-2\a-2}$. On the other hand, we know that
$$
K^{\a,\s}(x,y) \simeq\simeq x^{2\s-2\a-2} + (xy)^{-\a-1/2}\exp\Big(-c\frac{x}{y}\Big)
	\begin{cases}
		x^{2\s-1}, & \s< 1/2,\\
		1, & \s=1/2,\\
		(xy)^{\s-1/2}, & \s>1/2.
	\end{cases}
$$
This relation remains true after multiplying the exponential by an arbitrary power of $x/y$. Therefore
we see that the first term dominates in the above sum, and the conclusion follows.
\end{proof}

\begin{rem}
Simpler tools than Lemma \ref{lem:E} are sufficient for the proof of Theorem \ref{thm:ker},
see \cite[Lemma 2.1 and 2.2]{NoSt4}. However, we decided to use Lemma \ref{lem:E} since it allows 
for more compact notation and the proof.
\end{rem}

\begin{rem}
Theorem \ref{thm:ker} can be proved in another way, via expressing $K^{\a,\s}(x,y)$ by the
Hankel-Poisson kernel $P_t^{\a}(x,y)$, that is the integral kernel of the semigroup generated by means 
of the square root of $L_{\a}$. We have
$$
K^{\a,\s}(x,y) = \frac{1}{\Gamma(2\s)} \int_0^{\infty} P_t^{\a}(x,y)\,t^{2\s-1}\, dt.
$$
Using sharp estimates for $P_t^{\a}(x,y)$ from \cite[Theorem 6.1]{BHNV} and estimating the emerging integral
with the aid of \cite[Lemma 3.1]{NR} leads to the bounds asserted in Theorem \ref{thm:ker}.
Nevertheless, our approach based on the Hankel heat kernel is more direct.
\end{rem}

\begin{proof}[{Proof of Theorem \ref{thm:ker_bes}}]
Using \eqref{W_as} we write
\begin{align*}
H^{\a,\s}(x,y)  
& \simeq (xy)^{-\a-1/2}\int_0^{xy} e^{-t} \exp\bigg(-\frac{(x-y)^2}{4t}\bigg) t^{\s-3/2} \, dt
	+ \int_{xy}^{\infty} e^{-t} \exp\bigg( - \frac{x^2+y^2}{4t}\bigg) t^{\s-\a-2} \, dt \\
& \equiv I_0 + I_{\infty}.
\end{align*}
Changing the variables of integrations $t \mapsto xyt$ and $t \mapsto xy/t$, respectively, we get
$$
I_0 = (xy)^{\s-\a-1} E_{\s-3/2}\bigg( \frac{(x-y)^2}{4xy}, xy \bigg), \qquad
I_{\infty} = (xy)^{\s-\a-1} E_{\a-\s}\bigg( xy, \frac{x^2+y^2}{4xy}\bigg).
$$
Applying now Lemma \ref{lem:E} twice we arrive at the estimates
$$
I_0 \simeq \simeq (xy)^{-\a-1/2} \exp\bigg( -c \bigg[|x-y| \vee \frac{(x-y)^2}{xy}\bigg] \bigg)
	\begin{cases}
		|x-y|^{2\s-1}, & \s < 1/2, \\
		1 + \log^+\frac{1}{|x-y| \vee \frac{(x-y)^2}{xy}}, & \s=1/2,\\
		(1 \wedge xy)^{\s-1/2}, & \s > 1/2
	\end{cases} 
$$
and
$$
I_{\infty} \simeq \simeq \exp\big( -c \big[ (x+y) \vee xy \big] \big)
	\begin{cases}
		(x+y)^{2\s-2\a-2}, & \s < \a+1,\\
		1+\log^+\frac{1}{(x+y)\vee xy}, & \s=\a+1,\\
		1, & \s > \a +1.
	\end{cases}
$$
These bounds will be used in the sequel repeatedly, without further mention.

We first show (i) and thus assume to this end that $x+y \le 1$. We will inspect the cases of comparable and
non-comparable $x$ and $y$. For symmetry reasons, below we may always assume that $x \ge y$.\\
\noindent \textbf{Case 1: $\boldsymbol{y \le x \le 2y}$.} We must show that
\begin{equation} \label{ts1}
I_0 + I_{\infty} \simeq \chi_{\{\s > \a+1\}} + \chi_{\{\s=\a+1\}} \log\frac{1}x + x^{-2\a-1}
	\begin{cases}
		|x-y|^{2\s-1}, & \s<1/2,\\
		1+\log\frac{x}{|x-y|}, & \s=1/2,\\
		x^{2\s-1}, & \s > 1/2.
	\end{cases}
\end{equation}
Since $y \simeq x < 1$, it is easily seen that
\begin{align}
I_{\infty} & \simeq \chi_{\{\s > \a+1\}} + \chi_{\{\s=\a+1\}} \Big(1+ \log\frac{1}x\Big) + 
	\chi_{\{\s< \a+1\}}x^{2\s-2\a-2} \nonumber \\
& \simeq \chi_{\{\s > \a+1\}} + \chi_{\{\s=\a+1\}} \log\frac{1}x + x^{2\s-2\a-2}. \label{ww2}
\end{align}
As for $I_0$, we observe that for $x$ and $y$ under consideration
$$
|x-y| \vee \frac{(x-y)^2}{xy} \simeq |x-y| \vee \Big( 1- \frac{y}x\Big)^2 \lesssim 1
$$
and
$$	
1+\log^+\frac{1}{|x-y| \vee \frac{(x-y)^2}{xy}} \simeq 1 + \bigg(\log\frac{1}{|x-y|}\bigg) \wedge
	\bigg( \log \frac{x^2}{(x-y)^2}\bigg) \simeq 1 + \log\frac{x}{|x-y|}.
$$
Consequently,
$$
I_0 \simeq x^{-2\a-1}
	\begin{cases}
		|x-y|^{2\s-1}, & \s < 1/2,\\
		1+ \log\frac{x}{|x-y|}, & \s=1/2,\\
		x^{2\s-1}, & \s > 1/2.
	\end{cases}
$$
Combining the above bounds of $I_0$ and $I_{\infty}$ we get \eqref{ts1}.\\
\noindent \textbf{Case 2: $\boldsymbol{x > 2y}$.} Now the desired estimates take the form
$$
I_0 + I_{\infty} \simeq \chi_{\{\s > \a+1\}} + \chi_{\{\s=\a+1\}} \log\frac{1}x + x^{2\s-2\a-2}.
$$
As in the previous case, $I_{\infty}$ is comparable to the expression in \eqref{ww2},
so it suffices to check that $I_0$ is controlled by $I_{\infty}$.
Observe that for $2y < x < 1$
$$
I_0 \simeq \simeq (xy)^{-\a-1/2} \exp\Big(-c\frac{x}y\Big)
	\begin{cases}
		x^{2\s-1}, & \s \le 1/2,\\
		(xy)^{\s-1/2}, & \s > 1/2.
	\end{cases}
$$
This relation remains true if the right-hand side is multiplied by an arbitrary power of $x/y$, since
the ratio is at least $2$. Therefore
$$
I_0 \simeq \simeq \exp\Big(-c\frac{x}y\Big) x^{2\s-2\a-2}
$$
and the conclusion follows.

We pass to proving (ii), so from now on we consider $x+y>1$. Again, we may and do assume that $x \ge y$
and distinguish the cases of comparable and non-comparable arguments.\\
\noindent \textbf{Case 1: $\boldsymbol{y \le x \le 2y}$.} We aim at showing that
\begin{equation} \label{ts2}
I_0 + I_{\infty} \simeq \simeq x^{-2\a-1}\exp\big(-c|x-y|\big)
	\begin{cases}
		|x-y|^{2\s-1}, & \s < 1/2,\\
		1 + \log^+\frac{1}{|x-y|}, & \s = 1/2,\\
		1, & \s> 1/2.
	\end{cases}
\end{equation}
We have
$$
I_{\infty} \simeq \simeq \exp\big(-cx^2\big)
	\begin{cases}
		x^{2\s-2\a-2}, & \s < \a+1,\\
		1, & \s \ge \a+1,
	\end{cases}
$$
so $I_{\infty} \simeq \simeq \exp(-cx^2)$, which is controlled by the right-hand side in \eqref{ts2}.
On the other hand, for $x$ and $y$ under consideration
$$
|x-y| \vee \frac{(x-y)^2}{xy} \simeq |x-y|,
$$
therefore $I_0$ is comparable, in the sense of ``$\simeq \simeq$'', with the right-hand side of \eqref{ts2}.
Now \eqref{ts2} follows.\\
\noindent \textbf{Case 2: $\boldsymbol{x > 2y}$.} In this case the bounds to be verified are simply
\begin{equation} \label{ts3}
I_0 + I_{\infty} \simeq \simeq \exp(-cx).
\end{equation}
We have, for $x$ and $y$ satisfying $x+y>1$ and $x>2y$,
$$
I_0 \simeq \simeq (xy)^{-\a-1/2} \exp\big(-cx\big[1\vee y^{-1}\big]\big)
	\begin{cases}
		x^{2\s-1}, & \s \le 1/2,\\
		(1\wedge xy)^{\s-1/2}, & \s > 1/2
	\end{cases}
$$
and
$$
I_{\infty} \simeq \simeq \exp\big( -cx [1\vee y]\big)
	\begin{cases}
		x^{2\s-2\a-2}, & \s < \a+1,\\
		1, & \s \ge \a +1.
	\end{cases}
$$
If $y \ge 1$, then it is easily seen that $I_0 \simeq \simeq \exp(-cx)$ and 
$I_{\infty}\simeq \simeq \exp(-cxy)$,
so \eqref{ts3} follows. On the other hand, if $y < 1$ then $I_{\infty} \simeq \simeq \exp(-cx)$ and
$$
I_0 \simeq \simeq (xy)^{-\a-1/2} \exp\Big( -c \frac{x}{y}\Big)
	\begin{cases}
		x^{2\s-1}, & \s \le 1/2,\\
		(1 \wedge xy)^{\s-1/2}, & \s > 1/2.
	\end{cases}
$$
Multiplying the right-hand side above by $(x/y)^{-\a-1/2}$ (this does not change the estimates) and
using the bounds $x < x/y$ and $1\wedge xy \le 1$ we find that $I_0$ is in fact 
controlled by the right-hand side in \eqref{ts3}. The conclusion again follows.

The proof of Theorem \ref{thm:ker_bes} is complete.
\end{proof}

\subsection{Estimates of the Hankel-Dunkl potential kernels} \label{ssec:de}

We first focus our attention on the Dunkl heat kernel $\mathbb{W}_t^{\a}(x,y)$. Recall that this kernel
is given by means of the auxiliary function $\Phi_{\a}$. In \cite[Section 3.2]{NoSt5} it was shown
that
$$
\Phi_{\a}(u) \simeq u^{-\a}I_{\a}(u), \qquad u \ge 0,
$$
for $\a > -1$ (here the value of the right-hand side at $u=0$ is understood in a limiting sense), 
and when $\a > -1/2$
$$
\Phi_{\a}(u) \simeq |u|^{-\a} I_{\a}(|u|) \big( 1 \wedge |u|^{-1}\big), \qquad u < 0.
$$
Moreover, $\Phi_{\a}(u)$ is negative in case $\a<-1/2$ and $u<0$, provided that $|u|$ is sufficiently large.
Note the particular explicit case $\Phi_{-1/2}(u) = \sqrt{{2}/{\pi}} \exp(u)$ corresponding to
the classical one-dimensional Gauss-Weierstrass kernel $\mathbb{W}_t^{-1/2}(x,y)=W_t(x-y)$.

From the above properties of $\Phi_{\a}$ we conclude that $\mathbb{W}_t^{\a}(x,y)$ attains
negative values when $\a < -1/2$, $xy < 0$ and $|xy|/t$ is sufficiently large. 
Furthermore, with the
aid of the standard asymptotics \eqref{asym_I} for $I_{\a}$, we also get the following sharp and explicit
description of $\mathbb{W}_t^{\a}(x,y)$.
\begin{propo} \label{prop:Dhker}
Let $\a> -1/2$. The following estimates hold uniformly in $x,y \in \mathbb{R}$ and $t>0$.
\begin{itemize}
\item[(a)] If $xy \ge 0$, then
$$
\mathbb{W}_t^{\a}(x,y) \simeq \begin{cases}
												t^{-\a-1} \exp\left(-\frac{x^2+y^2}{4t}\right),& xy \le t,\\
												(xy)^{-\a-1/2}\frac{1}{\sqrt{t}}\exp\left(-\frac{(x-y)^2}{4t}\right), & xy > t.
											\end{cases}
$$
\item[(b)] If $xy < 0$, then
$$
\mathbb{W}_t^{\a}(x,y) \simeq \begin{cases}
												t^{-\a-1} \exp\left(-\frac{x^2+y^2}{4t}\right),& |xy| \le t,\\
												|xy|^{-\a-3/2}{\sqrt{t}}\exp\left(-\frac{(|x|-|y|)^2}{4t}\right), & |xy| > t.
											\end{cases}
$$
\end{itemize}
\end{propo}
Item (a) will not be needed in the sequel, nevertheless we state it for the sake of completeness.
We are now in a position to prove Theorem \ref{thm:ker_D}.
\begin{proof}[{Proof of Theorem \ref{thm:ker_D}}]
Since $\mathbb{K}^{-1/2,\s}(x,y)$ is the classical one-dimensional Riesz potential kernel, we look at
$\a > -1/2$.
Consider first $xy \ge 0$. Since $\mathbb{K}^{\a,\s}(x,y) = \mathbb{K}^{\a,\s}(-x,-y)$, we may assume that
$x,y \ge 0$. If $x,y>0$, then we easily get the desired estimates by means of \eqref{kkp2} and Theorem
\ref{thm:ker}. If $x=0$ or $y=0$, then \eqref{kkp2} still holds, with a limiting understanding of the
values of $K^{\a,\s}(x,y)$ and, implicitly, $W_t^{\a}(x,y)$. Tracing the proof of Theorem \ref{thm:ker},
one ensures that the asserted bounds for $K^{\a,\s}(x,y)$ remain true for all $x,y \ge 0$.
Hence the conclusion again follows. 

Now assume that $xy < 0$. By Proposition \ref{prop:Dhker} (b) we have
\begin{equation} \label{rel7}
\mathbb{K}^{\a,\s}(x,y) \simeq \widetilde{K}^{\a,\s}(|x|,|y|),
\end{equation}
where, for $\tilde{x},\tilde{y} > 0$,
\begin{align*}
\widetilde{K}^{\a,\s}(\tilde{x},\tilde{y}) & = 
(\tilde{x}\tilde{y})^{-\a-3/2}\int_0^{\tilde{x}\tilde{y}} t^{\s-1/2}\exp\bigg(
 -\frac{(\tilde{x}-\tilde{y})^2}{4t}\bigg)\, dt
	+ \int_{\tilde{x}\tilde{y}}^{\infty} t^{\s-\a-2} 
	\exp\bigg( -\frac{\tilde{x}^2+\tilde{y}^2}{4t}\bigg) \, dt \\
& \equiv I_0 + I_{\infty}. 
\end{align*}
Here $I_{\infty}$ agrees with $I_{\infty}$ in the proof of Theorem \ref{thm:ker}, and
$$
I_0  = 
(\tilde{x}\tilde{y})^{\s-\a-1} E_{\s-1/2}\bigg( \frac{(\tilde{x}-\tilde{y})^2}{4\tilde{x}\tilde{y}}, 0\bigg).
$$
Proceeding as in the proof of Theorem \ref{thm:ker}, we infer that 
$\widetilde{K}^{\a,\s}(\tilde{x},\tilde{y})\simeq (\tilde{x}+\tilde{y})^{2\s-2\a-2}$. 
This, in view of \eqref{rel7} and the relation
$|x|+|y| = |x-y|$ valid when $xy<0$, finishes the proof.
\end{proof}

To prove Theorem \ref{thm:ker_besD} it is necessary to obtain good estimates of the kernel
$$
\widetilde{H}^{\a,\s}(x,y) = \frac{1}{\Gamma(\s)} \int_0^{\infty} e^{-t}\mathbb{W}_t^{\a}(x,-y) t^{\s-1}\, dt,
	\qquad x,y  >0.
$$
\begin{lem} \label{lem:aker}
Let $\a > -1/2$. The following estimates hold uniformly in $x,y>0$.
\begin{itemize}
\item[(i)]
If $x + y \le 1$, then
$$
\widetilde{H}^{\a,\s}(x,y) \simeq \chi_{\{\s> \a+1\}} + \chi_{\{\s=\a+1\}} \log\frac{1}{x+y}
	+ (x+y)^{2\s-2\a-2}.
$$
\item[(ii)]
If $x+y > 1$, then
$$
\widetilde{H}^{\a,\s}(x,y) \simeq \simeq (x+y)^{-2\a-1} \exp\big(-c|x-y|\big) (x+y)^{-2}.
$$
\end{itemize}
\end{lem}

\begin{proof}
In view of Proposition \ref{prop:Dhker} (b),
\begin{align*}
\widetilde{H}^{\a,\s}(x,y)
& \simeq (xy)^{-\a-3/2}\int_0^{xy} e^{-t} \exp\bigg(-\frac{(x-y)^2}{4t}\bigg) t^{\s-1/2} \, dt
	+ \int_{xy}^{\infty} e^{-t} \exp\bigg( - \frac{x^2+y^2}{4t}\bigg) t^{\s-\a-2} \, dt \\
& \equiv I_0 + I_{\infty}.
\end{align*}
Here $I_{\infty}$ agrees with $I_{\infty}$ from the proof of Theorem \ref{thm:ker_bes}, and
$$
I_0 = (xy)^{\s-\a-1} E_{\s-1/2}\bigg( \frac{(x-y)^2}{4xy}, xy \bigg).
$$
Proceeding as in the proof of Theorem \ref{thm:ker_bes}, we arrive at the desired conclusion.
Details are left to the reader.
\end{proof}

\begin{proof}[{Proof of Theorem \ref{thm:ker_besD}}]
The reasoning is analogous to that in the proof of Theorem \ref{thm:ker_D}. Here instead of \eqref{kkp2}
and Theorem \ref{thm:ker} one uses \eqref{Bkkp2} and Theorem \ref{thm:ker_bes}, respectively.
The relevant estimate for the case of arguments having opposite signs is provided by Lemma \ref{lem:aker}.
\end{proof}

\section{$L^p-L^q$ estimates} \label{sec:LpLq}

In this section we prove all the $L^p-L^q$ results in the three settings investigated. 

\subsection{$\boldsymbol{L^p-L^q}$ estimates in the setting of the modified Hankel transform}
It is convenient to prove Theorem \ref{thm:main} first.
In the proof we will need the following characterization of two power-weight
$L^p-L^q$ inequalities for the Hardy operator and its dual.
\begin{lem} \label{lem:Hardy}
Let $A,B \in \mathbb{R}$ and let $1 \le p,q \le \infty$.
\begin{itemize}
\item[(a)]
The estimate
$$
\bigg\| x^B \int_0^x h(y)\, dy \bigg\|_{L^q(\mathbb{R}_+,dx)} 
	\lesssim \big\|x^A h\big\|_{L^p(\mathbb{R}_+,dx)}
$$
holds uniformly in 
$h \in L^p(\mathbb{R}_+,x^{Ap}dx)$
if and only if $p \le q$ and
$A - \frac{1}{p'} = B+\frac{1}q$ and $A< \frac{1}{p'}$ ($\le$ in case $p=q'=1$). 
\item[(b)]
The estimate
$$
\bigg\| x^B \int_x^{\infty} h(y)\, dy \bigg\|_{L^q(\mathbb{R}_+,dx)} 
	\lesssim \big\|x^A h\big\|_{L^p(\mathbb{R}_+,dx)}
$$
holds uniformly in 
$h \in L^p(\mathbb{R}_+, x^{Ap}dx)$
if and only if $p \le q$ and
$A - \frac{1}{p'} = B+\frac{1}q$ and 
$B>-\frac{1}q$ ($\ge$ in case $p=q'=1$). 
\end{itemize}
\end{lem}

\begin{proof}
The case $p \le q$ is contained in \cite[Theorems 1 and 2]{B} specified to power weights.
On the other hand, it seems to be well known, at least as a folklore, that the estimates in (a) and (b)
do not hold when $q < p$. In (a) the case $q<p<\infty$ can be easily concluded, for instance, from
\cite[Theorem 2.4]{SS} (see also references given there). The analogous fact in (b) follows by duality.
For $q<p=\infty$ a direct counterexample of $h(y) = y^{-A}$ does the job.
\end{proof}

Another tool we shall use is Young's inequality in the context of the multiplicative group
$G = (\mathbb{R}_+,\frac{dx}x)$ equipped with the natural convolution 
$f\star g(x) = \int_{\mathbb{R}_+} f(y)g(y^{-1}x)\, \frac{dy}y$, see e.g.\ \cite[Theorem 1.2.12]{G}.
Note that an extension of Young's inequality to weak type spaces (cf.\ \cite[Theorem 1.4.24]{G}), 
in the context of $G$ or $(\mathbb{R}\setminus \{0\},\frac{dx}{|x|})$, was one of the
main tools in \cite{DDD}.
\begin{lem}[Young's inequality] \label{lem:Young}
Let $1\le p,q,r \le \infty$ satisfy $\frac{1}q+1 = \frac{1}p + \frac{1}r$. Then for any $f \in L^p(G)$
and $g \in L^r(G)$ we have $f\star g \in L^q(G)$ and
$$
\|f\star g\|_{L^q(G)} \le \|g\|_{L^r(G)} \|f\|_{L^p(G)}.
$$
\end{lem}

We are now prepared to prove Theorem \ref{thm:main}.

\begin{proof}[{Proof of Theorem \ref{thm:main}}]
By Theorem \ref{thm:ker} we have
$$
K^{\a,\s}(x,y) \simeq (x+y)^{2\s-2\a-2} + \chi_{\{x/2<y<2x\}} y^{-2\a-1}
	\bigg[ \chi_{\{\s<1/2\}} |x-y|^{2\s-1} + \chi_{\{\s=1/2\}} \log\frac{2(x+y)}{|x-y|} \bigg],
$$
uniformly in $x,y > 0$. Thus we can write the estimates
\begin{equation} \label{tse}
I^{\a,\s} f \simeq H_0f + H_{\infty}f + \chi_{\{\s< 1/2\}} Tf + \chi_{\{\s=1/2\}} Sf, \qquad f \ge 0,
\end{equation}
where the relevant operators are defined as follows:
\begin{align*}
H_0f(x)  = x^{2\s-2\a-2} \int_0^x f(y)\, d\mu_{\a}(y),& \qquad 
H_{\infty}f(x)  = \int_x^{\infty} y^{2\s-2\a-2}f(y) \, d\mu_{\a}(y), \\
Tf(x)  = \int_{x/2}^{2x} |x-y|^{2\s-1} f(y)\, dy, &  \qquad
Sf(x)  = \int_{x/2}^{2x} \log\frac{2(x+y)}{|x-y|} f(y)\, dy.
\end{align*}
Clearly, if each term of the right-hand side in \eqref{tse} is well defined
(i.e.\ the defining integrals converge for a.a.\ $x>0$) for a fixed, not necessarily non-negative $f$, 
then so is $I^{\a,\s}$. On the other hand, if any of these terms is not well defined for an $f \ge 0$, 
then neither is $I^{\a,\s}f$. Similar implications pertain to weighted $L^p-L^q$ mapping properties.

We first prove (i). Let $f \in L^p(x^{ap}d\mu_{\a})$. 
By means of H\"older's inequality
it is straightforward to verify that $H_0 f$ is well defined when $a< \frac{2\a+2}{p'}$ ($\le$ if $p=1$)
and $H_{\infty}f$ is well defined when $a > 2\s-\frac{2\a+2}{p}$ ($\ge$ if $p=1$). These conditions are
sharp in the sense that if $a$ is beyond the indicated ranges, then there exists a function 
$g \in L^p(x^{ap}d\mu_{\a})$ such that 
$H_0 g(x) = \infty$, $x > 0$, or $H_{\infty}g(x)=\infty$, $x>0$, respectively.
Essentially, the simplest examples of such $g$ are the following. In case of $H_0$,
$g(y) = \chi_{\{y < 1\}}y^{-2\a-2}$ when
$a>\frac{2\a+2}{p'}$ ($\ge$ if $p=\infty$), and 
$g(y) = \chi_{\{y < 1\}}y^{-2\a-2}/\log\frac{2}y$ when $a=\frac{2\a+2}{p'}$
and $1< p < \infty$. In case of $H_{\infty}$, 
$g(y) = \chi_{\{y>2\}}y^{-2\s}$ when $a < 2\s - \frac{2\a+2}p$ ($\le$ if $p=\infty$), and
$g(y) = \chi_{\{y>2\}} y^{-2\s}/\log y$ when $a=2\s-\frac{2\a+2}{p}$ and $1< p < \infty$. Altogether,
this shows that condition \eqref{cnd17}
is necessary and sufficient for the sum $H_0f + H_{\infty}f$ to be well defined.

Now it suffices to ensure that $Tf$ and $Sf$ are well defined under \eqref{cnd17}.
But even more is true, since in fact no restrictions on $p$ and $a$ are needed.
Indeed, let $0 \le f \in L^p(x^{ap}d\mu_{\a})$ 
with arbitrary 
$a \in \mathbb{R}$ and $1\le p \le \infty$. Consider $\tilde{f} = f\chi_{(1/n,n)}$, $n$ large and fixed.
Since $Tf(x) = T\tilde{f}(x)$ for $x \in (2/n, n/2)$, it is enough to check that $T\tilde{f}$ is well
defined, and similarly in case of $S$. Observe that $\tilde{f} \in L^p(\mathbb{R}_{+},dx)$ and since
its support is bounded, also $\tilde{f} \in L^1(\mathbb{R}_+,dx)$. Then using the Fubini-Tonelli theorem
we get
$$
\|T\tilde{f}\|_{L^1(\mathbb{R}_+,dx)} = \int_0^{\infty}\int_{y/2}^{2y} |x-y|^{2\s-1}\, dx \, \tilde{f}(y)\,dy
	\simeq \int_0^{\infty} y^{2\s} \tilde{f}(y)\, dy \lesssim \|\tilde{f}\|_{L^1(\mathbb{R}_+,dx)} < \infty.
$$
The case of $S$ is analogous, a simple computation leads to
$$
\|S\tilde{f}\|_{L^1(\mathbb{R}_+,dx)} = \int_0^{\infty}\int_{y/2}^{2y} \log\frac{2(x+y)}{|x-y|}\, dx
	\, \tilde{f}(y)\, dy \simeq \int_0^{\infty} y \tilde{f}(y)\, dy < \infty.
$$
The conclusion follows.

We pass to proving (ii). We will analyze separately each of the four terms on the right-hand side of
\eqref{tse}. Altogether, this will imply (ii). Since all the considered operators are non-negative,
below we may and always do assume that $f \ge 0$.\\
\textbf{Analysis of $\boldsymbol{H_0}$.} 
Substituting $f(y) = y^{-2\a-1}h(y)$ we see that the estimate
$$
\|x^{-b}H_0 f\|_{L^q(d\mu_{\a})} \lesssim \|x^a f\|_{L^p(d\mu_{\a})}
$$
is equivalent to
$$
\bigg\| x^{-b+(2\a+1)/q} \int_0^x h(y) \, dy\bigg\|_{L^q(\mathbb{R}_+,dx)} 
	\lesssim \big\| x^{a+(2\a+1)/p-2\a-1}h\big\|_{L^p(\mathbb{R}_+,dx)}.
$$
Applying now Lemma \ref{lem:Hardy} (a) (specified to $A=a-\frac{2\a+1}{p'}$ and 
$B=-b + \frac{2\a+1}{q}+2\s - 2\a-2$) we conclude that this holds if and only if $p \le q$ and
$\frac{1}q = \frac{1}p + \frac{a+b-2\s}{2\a+2}$ and $a < \frac{2\a+2}{p'}$ ($\le$ in case $p=q'=1$).
The latter three are precisely conditions (a), (b) and (c) of the theorem.\\
\textbf{Analysis of $\boldsymbol{H_{\infty}}$.}
Substituting $f(y) = y^{1-2\s}h(y)$ we can write the estimate
$$
\|x^{-b}H_{\infty} f\|_{L^q(d\mu_{\a})} \lesssim \|x^a f\|_{L^p(d\mu_{\a})}
$$
in the equivalent form
$$
\bigg\| x^{-b+(2\a+1)/q} \int_x^{\infty}h(y)\,dy\bigg\|_{L^q(\mathbb{R}_+,dx)} 
	\lesssim \big\| x^{a+(2\a+1)/p+1-2\s}h\big\|_{L^p(\mathbb{R}_+,dx)}.
$$
Using Lemma \ref{lem:Hardy} (b) (with $A=a+\frac{2\a+1}p + 1-2\s$ and $B=-b + \frac{2\a+1}q$) we infer
that this holds if and only if $p \le q$ and $\frac{1}q = \frac{1}p + \frac{a+b-2\s}{2\a+2}$
and $b < \frac{2\a+2}q$ ($\le$ in case $p=q'=1$). These are conditions (a), (b) and (d) of the theorem.\\
\textbf{Analysis of $\boldsymbol{T}$ in case $\boldsymbol{\s<1/2}$.} 
Here we may assume that (a)-(d) are satisfied. 
Observe that, in view of condition (b),
$$
a+\frac{2\a+1}p + b - \frac{2\a+1}q
= 2\s + \frac{1}q - \frac{1}p.
$$
Thus, letting $f(y) = g(y) y^{-a-(2\a+1)/p}$, we see that the estimate
$$
\|x^{-b}T f\|_{L^q(d\mu_{\a})} \lesssim \|x^a f\|_{L^p(d\mu_{\a})}
$$
can be restated as
\begin{equation} \label{loc1}
\big\| T\big(y^{-2\s-1/q+1/p}g\big)\big\|_{L^q(\mathbb{R}_+,dx)} \lesssim \|g\|_{L^p(\mathbb{R}_+,dx)}.
\end{equation}

We claim that condition (e) is necessary for \eqref{loc1} to hold. 
Indeed, assume that $\frac{1}q < \frac{1}p-2\s$
and take $g(y) = \chi_{(1/2,1)}(y)\, (1-y)^{\gamma}$ with $\gamma = -\frac{1}p+\varepsilon$, where
$0 < \varepsilon < \frac{1}p-2\s-\frac{1}q$. Then the right-hand side of \eqref{loc1} is finite. 
On the other hand, for $x \in (1,3/2)$
\begin{align*}
\int_{x/2}^{2x} |x-y|^{2\s-1} \chi_{(1/2,1)}(y)\,(1-y)^{\gamma} y^{-2\s-1/q+1/p}\, dy & \gtrsim
\int_{3/4}^1 |x-y|^{2\s-1} (1-y)^{\gamma}\, dy \\
& = (x-1)^{2\s+\gamma} \int_0^{1/(4(x-1))} (1+u)^{2\s-1} u^{\gamma}\, du.
\end{align*}
Since here $1/(4(x-1))> 1/2$, the last integral is larger than a positive constant and so the 
left-hand side in \eqref{loc1} is larger than the constant times
$$
\big\| \chi_{(1,3/2)}(x)\,(x-1)^{2\s+\gamma}\big\|_{L^{\infty}(dx)}.
$$
But this expression is infinite since $2\s+\gamma < -\frac{1}{q}$, so \eqref{loc1} does not hold.
When $p=1$ and $\frac{1}q=1-2\s$, the counterexample of 
$g(y) = \chi_{(1/2,1)}(y)\, (1-y)^{-2\s}/\log\frac{2}{1-y}$ shows in a similar manner that $T$
is not bounded from $L^{1/(2\s)}(\mathbb{R}_+,dx)$ to $L^{\infty}(\mathbb{R}_+,dx)$,
and then, by duality, neither bounded from $L^1(\mathbb{R}_+,dx)$ to $L^{1/(1-2\s)}(\mathbb{R}_+,dx)$.
The claim follows.

Next, we prove that condition (e) is sufficient for \eqref{loc1}. When $\frac{1}q=\frac{1}p-2\s$, this
is an obvious consequence of the classical Hardy-Littlewood-Sobolev theorem in dimension one. So it remains
to verify \eqref{loc1} in case $\frac{1}q > \frac{1}p - 2\s$.
For this purpose we let $g(x) = F(x) x^{-1/p}$ and write \eqref{loc1} as
\begin{equation} \label{loc2}
\bigg\|\int_{x/2}^{2x} |y^{-1}x-1|^{2\s-1}\, F(y) \,\frac{dy}y\bigg\|_{L^q(\mathbb{R}_+,\frac{dx}x)}
\lesssim \|F\|_{L^p(\mathbb{R}_+,\frac{dx}x)}. 
\end{equation}
This is precisely $L^p-L^q$ estimate for the convolution operator on the multiplicative group
$(\mathbb{R}_+,\frac{dx}x)$ given by the convolution kernel
$$
K(u) = \chi_{(1/2,2)}(u)\, |u-1|^{2\s-1}, \qquad u > 0.
$$
Notice that $K \in L^r(\mathbb{R}_+,\frac{dx}x)$ if (and only if) 
$\frac{1}r > 1-2\s$. Now, with the aid of Lemma \ref{lem:Young} we readily arrive at the desired conclusion.\\
\textbf{Analysis of $\boldsymbol{S}$ in case $\boldsymbol{\s=1/2}$.}
When $\s=1/2$ condition (e) says that $(p,q) \neq (1,\infty)$. Assuming that, the estimate 
we must prove is equivalent to
$$
\bigg\|\int_{x/2}^{2x} \log\frac{2(y^{-1}x+1)}{|y^{-1}x-1|} 
	\, F(y) \,\frac{dy}y\bigg\|_{L^q(\mathbb{R}_+,\frac{dx}x)}
	\lesssim \|F\|_{L^p(\mathbb{R}_+,\frac{dx}x)}. 
$$
Now the relevant convolution kernel is
$$
K(u) = \chi_{(1/2,2)}(u)\,  \log\frac{2(u+1)}{|u-1|}, \qquad u > 0.
$$
Since $K \in L^r(\mathbb{R}_+,\frac{dx}x)$ for all $1 \le r < \infty$, 
the conclusion follows by Lemma \ref{lem:Young}.

On the other hand, the estimate
$$
\|x^{-b}S f\|_{L^{\infty}(d\mu_{\a})} \lesssim \|x^a f\|_{L^1(d\mu_{\a})}
$$
does not hold. To see this, consider $f(y) = \chi_{(1/2,1)}(y)/((1-y)\log^2\frac{2}{1-y})$.
Then $\|x^a f\|_{L^1(d\mu_{\a})}$ is finite, but
$$
\|x^{-b} Sf\|_{L^{\infty}(d\mu_{\a})} \gtrsim \essup_{1/2<x<1} Sf(x) \ge 
	\int_{1/2}^1 \log\frac{2(1+y)}{1-y} f(y)\, dy \ge \int_{1/2}^1 \frac{dy}{(1-y)\log\frac{2}{1-y}} = \infty.
$$

The proof of Theorem \ref{thm:main} is now complete.
\end{proof}

We turn to proving Theorem \ref{thm:LpLq}. Items (i) and (ii) follow immediately from 
Theorem \ref{thm:main} specified to $a=b=0$, so it remains to prove (iii). 
\begin{proof}[{Proof of Theorem \ref{thm:LpLq} (iii)}]

Assume that $\a \ge -1/2$ and fix $q = \frac{\a+1}{\a+1-\s}$. 
We will prove that $I^{\a,\s}$ is of weak type $(1,q)$.
Because of \eqref{tse}, it suffices to show the weak type $(1,q)$ of the operators $H_0+H_{\infty}$,
$T$ in case $\s< 1/2$ and $S$ in case $\s=1/2$.

Treatment of $H_0+H_{\infty}$ is straightforward. We have
$$
(H_0+H_{\infty})f(x) \lesssim \int_0^{\infty} (x+y)^{2\s-2\a-2} f(y)\, d\mu_{\a}(y), \qquad x > 0,
$$
uniformly in $f \ge 0$. Since 
$\|(x+\cdot)^{2\s-2\a-2}\|_{L^{\infty}(\mathbb{R}_+)} \simeq x^{2\s-2\a-2}$, $x>0$, we get
$$
|(H_0+H_{\infty})f(x)| \lesssim x^{2\s-2\a-2} \|f\|_{L^1(d\mu_{\a})}, \qquad x >0.
$$
As easily verified, the function $x \mapsto x^{2\s-2\a-2}$ belongs to weak $L^q(d\mu_{\a})$ and the desired
conclusion follows.

Next, let $\s<1/2$ and consider the operator $T$. In case $\a=-1/2$ the weak type $(1,\frac{1}{1-2\s})$
follows from the analogous well-known result for the classical one-dimensional Riesz potential.
Thus we may assume $\a > -1/2$. We claim that $T$ is even bounded from $L^1(d\mu_{\a})$ to $L^q(d\mu_{\a})$,
$q=\frac{\a+1}{\a+1-\s}$. To prove the claim, we show that $T$ is bounded from $L^{q'}(d\mu_{\a})$,
$q'=\frac{\a+1}{\s}$, to $L^{\infty}(d\mu_{\a})$. This is enough, because the kernel of $T$ is symmetric
and $L^1(d\mu_{\a}) \subset (L^1(d\mu_{\a}))^{**} = (L^{\infty}(d\mu_{\a}))^*$. 
By H\"older's inequality and the change of variable $y/x=u$ we obtain
\begin{align*}
|Tf(x)| & \lesssim \|f\|_{L^{q'}(d\mu_{\a})} \bigg( \int_{x/2}^{2x} \big( y^{-(2\a+1)/q'}
	|x-y|^{2\s-1}\big)^q\, dy\bigg)^{1/q} \\
& = \|f\|_{L^{q'}(d\mu_{\a})} \bigg(\int_{1/2}^{2} 
	u^{-(2\a+1)q/q'} |1-u|^{(2\s-1)q}\, du\bigg)^{1/q}.
\end{align*}
Since $(2\s-1)q > -1$ when $\a > -1/2$, the last integral is finite. 
We see that $|Tf(x)| \lesssim \|f\|_{L^{q'}(d\mu_{\a})}$, $x>0$, and the claim follows. 

Finally, consider the operator $S$ in case $\s=1/2$. 
We will check that $S$ is even bounded from $L^1(d\mu_{\a})$ to $L^q(d\mu_{\a})$, $q=\frac{\a+1}{\a+1/2}$.
Proceeding as in case of $T$ above we arrive at
\begin{align*}
|Sf(x)| & \lesssim \|f\|_{L^{q'}(d\mu_{\a})} \bigg( \int_{x/2}^{2x} \bigg( y^{-(2\a+1)/q'}
	\log\frac{2(x+y)}{|x-y|}\bigg)^q\, dy\bigg)^{1/q} \\
& = \|f\|_{L^{q'}(d\mu_{\a})} \bigg(\int_{1/2}^{2} 
	\Big(u^{-(2\a+1)/q'}\log\frac{2(1+u)}{|1-u|}\Big)^{q}\, du\bigg)^{1/q}.
\end{align*}
Here the last integral is finite, so $|Sf(x)| \lesssim \|f\|_{L^{q'}(d\mu_{\a})}$, $x>0$, and
the conclusion follows.

To finish the proof, we need to ensure that $I^{\a,\s}$ is not of weak type $(1,\frac{\a+1}{\a+1-\s})$
when $\a < -1/2$. This follows by an \emph{au contraire} argument. If $I^{\a,\s}$ were
of weak type $(1,\frac{\a+1}{\a+1-\s})$, then by duality 
(see the proof of \cite[Theorem 3.1]{NoSt4} or item (C) in the beginning of \cite[Section 4]{NR}) 
it would be of restricted weak type $(\frac{\a+1}{\s},\infty)$
and then, by interpolation, bounded from $L^p(d\mu_{\a})$ to $L^q(d\mu_{\a})$ for $p$ and $q$ satisfying
$\frac{1}q=\frac{1}p-\frac{\s}{\a+1}$ and $1 < p < \frac{\a+1}{\s}$. 
But this contradicts the necessity part of item (ii).
\end{proof}

Passing to the proof of Theorem \ref{thm:LpLq_bes}, it is easily seen that it
follows from Lemmas \ref{lem:bes_loc} and \ref{lem:bes_glob}
below that describe sharply $L^p-L^q$ behavior of two auxiliary operators (with non-negative kernels)
into which $J^{\a,\s}$ splits naturally. These operators are interesting
in their own right, so 
the lemmas provide slightly more information than
actually needed to conclude Theorem \ref{thm:LpLq_bes}.

We split $J^{\a,\s}$ according to the kernel splitting
$$
H^{\a,\s}(x,y)  = \chi_{\{x\le 2,y\le 2\}} H^{\a,\s}(x,y) + \chi_{\{x \vee y > 2\}}H^{\a,\s}(x,y)
     \equiv H^{\a,\s}_0(x,y) + H^{\a,\s}_\infty(x,y)
$$
and denote the resulting integral operators by $J^{\a,\s}_0$ and $J^{\a,\s}_\infty$, respectively.
\begin{lem}[{\cite[Lemma 4.1]{NoSt5} and \cite{NR}}] \label{lem:bes_loc}
Let $\a > -1$, $\s>0$ and $1 \le p,q \le \infty$. Set
$\delta:= ((-1/2) \vee \a) + 1$. Then $J_0^{\a,\s}$ is bounded from $L^p(d\mu_{\a})$ to $L^q(d\mu_{\a})$
if and only if
$$
\frac{1}{p} - \frac{\s}{\delta} \le \frac{1}q \qquad \textrm{and} \qquad
\bigg(\frac{1}p,\frac{1}q\bigg) \notin \bigg\{
\Big(\frac{\s}{\delta},0\Big), \Big(1,1-\frac{\s}{\delta}\Big)\bigg\}.
$$
\end{lem}
\begin{lem} \label{lem:bes_glob}
Let $\a > -1$, $\s>0$ and $1 \le p,q \le \infty$. 
Then $J^{\a,\s}_\infty$ is bounded from $L^p(d\mu_{\a})$ to $L^q(d\mu_{\a})$
if and only if $\a < -1/2$ and $p=q$ or $\a\ge -1/2$ and
$$
\frac{1}p - 2\s \le \frac{1}q \le \frac{1}p \qquad \textrm{and} \qquad
	\bigg(\frac{1}p,\frac{1}q\bigg) \notin \big\{(2\s,0), (1,1-2\s)\big\}.
$$
\end{lem}

In view of Theorem \ref{thm:ker_bes}, the behavior of $H^{\a,\s}_0(x,y)$ is exactly the same as the
behavior of the analogous kernel in the setting of Laguerre function expansions of convolution type
that was considered in \cite[Section 4.1]{NoSt5}.
Consequently, Lemma \ref{lem:bes_loc} coincides with \cite[Lemma 4.1]{NoSt5} and hence it follows from
the results of Nowak and Roncal \cite{NR}, as commented in \cite[Section 4.1]{NoSt5}.
To prove Lemma \ref{lem:bes_glob} we will need the following technical result. 
\begin{lem} \label{lem:ass}
Let $\a > -1$ and $\s>0$. Then the estimates
\begin{equation}  \label{aa}
\|H^{\a,\s}_\infty(x,\cdot)\|_{L^p(d\mu_{\a})}\simeq (1\vee x)^{(2\a+1)(1\slash p-1)}, \qquad x>0,
\end{equation}
hold for $1 \le p \le \infty$ when $\s>1/ 2$ and for $1 \le p < \frac{1}{1-2\s}$ when $\s\le 1/2$.

Moreover, for $\sigma\le 1/2$ and $\frac{1}{1-2\s} \le p \le \infty$, we have
\begin{equation}  \label{bb}
\|H^{\a,\s}_\infty(x,\cdot)\|_{L^p(d\mu_{\a})} = \infty, \qquad x > 4.
\end{equation}
\end{lem}
Actually, only \eqref{aa} will be used in the sequel.  However, we include also \eqref{bb} to show that
\eqref{aa} is optimal in the sense of the ranges of admissible parameters.

\begin{proof}[{Proof of Lemma \ref{lem:ass}}]
By Theorem \ref{thm:ker_bes}, $H^{\a,\s}_\infty(x,y)$ satisfies the estimates of
Theorem \ref{thm:ker_bes} (ii) outside the square $0 < x,y \le 2$, and vanishes inside this square.
Therefore, it is convenient to consider separately the cases $\s<1\slash2$, $\s=1\slash2$ and $\s>1\slash2$.
In what follows we treat the case $\s<1\slash2$ leaving a similar analysis for the remaining cases to the
reader.
We also observe that considering $0< x< 1$ and $x>4$ is enough for the proof of \eqref{aa},
see the proof of \cite[Lemma 4.3]{NoSt5}.

Let $\s<1\slash2$. In view of Theorem \ref{thm:ker_bes} (ii),
$$
H_{\infty}^{\a,\s}(x,y) \simeq\simeq \chi_{\{x \vee y > 2\}} (x+y)^{-2\a-1}|x-y|^{2\s-1}
        \exp\big( -c|x-y|\big).
$$
Hence, if $0<x<1$, then
$$
\int_2^\infty H^{\a,\s}_\infty(x,y)^p\,y^{2\a+1}dy \simeq \simeq
\int_2^\infty y^{(2\a+1)(1-p)+(2\s-1)p}\exp\big(-cpy\big)\,dy \simeq 1
$$
for $p < \infty$, and
$$
\essup_{y>2} H^{\a,\s}_\infty(x,y)\simeq \simeq \sup_{y>2} y^{2\s-2\a-2}\exp(-cy) \simeq 1.
$$
Thus \eqref{aa} for $x<1$ follows.
If $x>4$, then for $p < \frac{1}{1-2\s}$ and for the decisive interval $(x\slash2,3x\slash2)$ we have
\begin{align*}
\int_{x\slash2}^{3x\slash2} H^{\a,\s}_\infty(x,y)^p\,y^{2\a+1}dy &
\simeq \simeq x^{(2\a+1)(1-p)}\int_{x\slash2}^{3x\slash2} \exp(-cp|x-y|)|x-y|^{(2\s-1)p}\,dy \\
&= 2  x^{(2\a+1)(1-p)}\int_{0}^{x\slash2} \exp(-cpu)u^{(2\s-1)p}\,du\\
&\simeq  x^{(2\a+1)(1-p)}.
\end{align*}
Notice that the assumption imposed on $p$ guarantees convergence of the last integral.
Checking that the relevant integrals over $(0,x/2)$ and $(3x/2,\infty)$ are controlled by
$x^{(2\a+1)(1-p)}$ is straightforward. Now \eqref{aa} follows.

If $\frac{1}{1-2\s} \le p < \infty$, then the above argument leads also to \eqref{bb}.
Finally, we have
$$
\|H^{\a,\s}_\infty(x,\cdot)\|_{\infty} \ge \essup_{x\slash2<y<3x\slash2} H^{\a,\s}_\infty(x,y) 
 \simeq \simeq x^{-2\a-1}\essup_{x\slash2<y<3x\slash2} \exp(-c|x-y|)|x-y|^{2\s-1}=\infty,
$$
which justifies \eqref{bb} for $p=\infty$.
\end{proof}

\begin{proof}[{Proof of Lemma \ref{lem:bes_glob}}]
The structure of the proof is as follows.
The upper estimate of Lemma~\ref{lem:ass} readily enables us to establish
$L^1-L^q$
boundedness of ${J}^{\a,\s}_\infty$ for the relevant $q$. 
This, together with a duality argument based on the symmetry of the kernel,
$H^{\a,\s}_\infty(x,y)=H^{\a,\s}_\infty(y,x)$, and the Riesz-Thorin interpolation theorem, gives
$L^p-L^q$ bounds for $p$ and $q$ satisfying $p=q$ in case $\a < -1/2$ or
$$
\frac{1}p-{2\s} < \frac{1}q \le \frac{1}p
$$
when $\a \ge -1/2$.
The case when $\a \ge -1/2$ and $\frac{1}q = \frac{1}p-2\s$ and $2\s < \frac{1}p < 1$ follows
from Theorem~\ref{thm:LpLq}~(ii) and the fact that $H_{\infty}^{\a,\s}(x,y)$ is dominated pointwise by
$K^{\a,\s}(x,y)$.
Finally, the lack of $L^p-L^q$ boundedness
for the relevant $p$ and $q$ will be shown by indicating explicit counterexamples.
To simplify the notation,
in what follows $\|\cdot\|_p$ denotes the norm in the Lebesgue space $L^p(\mathbb{R}_+,d\mu_{\a})$.

The $L^1-L^q$ boundedness of ${J}^{\a,\s}_\infty$ holds for
$$
q\in
    \begin{cases}
        [1,\infty], & \s>1/2,\\
        [1,\frac1{1-2\s}), & \s\le1/2,
    \end{cases}
    \qquad {\rm or} \qquad
    q = 1,
$$
when $\a\ge-1/2$ or $-1<\a<-1/2$, respectively. Indeed, by Minkowski's integral inequality
(naturally extended to the case $q=\infty$), we get
$$
\|{J}^{\a,\s}_\infty f\|_q\le \|f\|_1\big\|\|H^{\a,\s}_\infty(\cdot,y)\|_q\big\|_{\infty}
$$
and the assertion follows provided that $\big\|\|H^{\a,\s}_\infty(\cdot,y)\|_q\big\|_{\infty}<\infty$
(the outer norms are always taken with respect to the $y$ variable).
For $q=\infty$ this is the case if $\a \ge -1/2$ and $\s > 1/2$, since then, by Lemma~\ref{lem:ass},
$$
\big\|\|H^{\a,\s}_\infty(\cdot,y)\|_\infty\big\|_{\infty}
\lesssim \sup_{y>0}(1\vee y)^{-(2\a+1)}<\infty.
$$
On the other hand, for $1\le q<\infty$ in case $\s>1/2$, or for $1\le q<\frac{1}{1-2\s}$ in case $\s\le1/2$
(so that Lemma~\ref{lem:ass} can be applied),
$$
\big\|\|H^{\a,\s}_\infty(\cdot,y)\|_q\big\|_{\infty}
\lesssim \sup_{y>0}(1\vee y)^{(2\a+1)(1/q-1)}<\infty,
$$
provided that $(2\a+1)(\frac{1}q-1) \le 0$, 
and this happens if $q$ satisfies the imposed restrictions.

We now use the fact that, due to the symmetry of the kernel and a duality argument,
$L^p-L^q$ boundedness of $J^{\a,\s}_\infty$ for some $1 \le p,q< \infty$ implies
$L^{q'}-L^{p'}$ boundedness of $J^{\a,\s}_\infty$. This allows us to infer from the results
already obtained that ${J}^{\a,\s}_\infty$ is $L^p-L^{\infty}$ bounded provided that
$$
p\in
    \begin{cases}
        (1,\infty], & \s>1/2,\\
        (\frac1{2\s},\infty], & \s\le1/2,
    \end{cases}   
    \qquad {\rm or} \qquad
    p = \infty,
$$
when $\a\ge-1/2$ or $-1<\a<-1/2$, respectively. Applying the Riesz-Thorin interpolation theorem
we conclude $L^p-L^q$ boundedness of $J^{\a,\s}_{\infty}$ in all the relevant cases

Passing to the negative results, we must verify the following items.
\begin{itemize}
\item[(a)] $J^{\a,\s}_{\infty}$ is not $L^p-L^q$ bounded when $p > q$.
\item[(b)] $J^{\a,\s}_{\infty}$ is not $L^p-L^q$ bounded when $p < q$ and $\a < -1/2$.
\item[(c)] $J^{\a,\s}_{\infty}$ is not $L^p-L^q$ bounded when $\frac{1}q < \frac{1}p - 2\s$ and $\s< 1/2$
	and $\a \ge -1/2$.
\item[(d)] $J^{\a,\s}_{\infty}$ is not $L^p-L^q$ bounded for 
	$(\frac{1}p,\frac{1}q) \in \{(2\s,0),(1,1-2\s)\}$
	when $\s \le 1/2$ and $\a \ge -1/2$.
\end{itemize}

To show (a), consider first $p=\infty$. Then, by Lemma \ref{lem:ass},
$$
\| J^{\a,\s}_{\infty}\boldsymbol{1}\|_{q}^q \simeq \int_0^{\infty} d\mu_{\a}(x) = \infty,
$$
hence $J^{\a,\s}_{\infty}$ is not $L^{\infty}-L^q$ bounded.
To treat the case $p < \infty$, take $f(y) = \chi_{\{y > 2\}}y^{-\xi}$, where $\xi>0$ satisfies
$\frac{2(\a+1)}p < \xi \le \frac{2(\a+1)}{q}$. Then $f \in L^p(d\mu_{\a})$. We claim that
$J_{\infty}^{\a,\s}f \notin L^q(d\mu_{\a})$. Indeed, using the lower bound from 
Theorem \ref{thm:ker_bes} (ii), we obtain
\begin{align*}
J_{\infty}^{\a,\s}f(x) & \ge f(x) \int_{x/2}^x H_{\infty}^{\a,\s}(x,y)\, d\mu_{\a}(y)\\
	& \gtrsim f(x) \int_{x/2}^x \exp(-c(x-y))
	\left\{ 
	\begin{array}{ll}
		(x-y)^{2\s-1}, & \s < 1/2 \\
		1+\log^+\frac{1}{x-y}, & \s=1/2 \\
		1, & \s > 1/2
	\end{array}
	\right\} \, dy \\
	& \simeq f(x),
\end{align*}
where the last relation follows by the change of variable $x-y=u$. 
Since $f \notin L^q(d\mu_{\a})$, the claim follows.

To justify (b), consider $f_n = \chi_{(n,n+1)}$ with $n$ large. Then $\|f_n\|_{p} \simeq n^{(2\a+1)/p}$.
Moreover, by Theorem~\ref{thm:ker_bes}~(ii), for $x \in (n,n+1)$
$$
J_{\infty}^{\a,\s}f_n(x) \gtrsim \int_n^{n+1} \exp(-c(x-y))
	\left\{ 
	\begin{array}{ll}
		|x-y|^{2\s-1}, & \s < 1/2 \\
		1+\log^+\frac{1}{|x-y|}, & \s=1/2 \\
		1, & \s > 1/2
	\end{array}
	\right\} \, dy
\simeq 1.
$$
This implies that
$$
\|J_{\infty}^{\a,\s}f_n\|_{q} \gtrsim \bigg( \int_{n}^{n+1} d\mu_{\a}(x)\bigg)^{1/q} \simeq n^{(2\a+1)/q}.
$$ 
Now, if $J_{\infty}^{\a,\s}$ were $L^p-L^q$ bounded, then we would have
$$
n^{(2\a+1)/q} \lesssim \|J_{\infty}^{\a,\s}f_n\|_{q} \lesssim \|f_n\|_p \simeq n^{(2\a+1)/p}, \qquad
	n \to \infty.
$$
But this is not possible since $2\a+1< 0$ and $p < q$.

Items (c) and (d) are proved by means of Theorem \ref{thm:ker_bes} (ii) and the counterexamples
presented in connection with items (b) and (c) in the proof of \cite[Lemma 4.2]{NoSt5}. 
We leave further details to
interested readers.
\end{proof}

\subsection{$\boldsymbol{L^p-L^q}$ estimates in the setting of the non-modified Hankel transform}
Let us first check that
Theorem \ref{thm:main_H} follows in a straightforward manner from Theorem \ref{thm:main}.
\begin{proof}[{Proof of Theorem \ref{thm:main_H}}]
By means of \eqref{dd} we see that 
$$
L^p(x^{ap}dx) \subset \domain \mathcal{I}^{\a,\s} \quad \textrm{if and only if} \quad
L^p\big(x^{[a+\a+1/2-(2\a+1)/p]p}d\mu_{\a}\big)\subset \domain I^{\a,\s}.
$$
Thus (i) follows from Theorem \ref{thm:main} (i).

To show (ii), notice that, in view of \eqref{pot_m}, the estimate in question is equivalent to the bound
$$
\big\| x^{-b+\a+1/2-(2\a+1)/q} I^{\a,\s}g\|_{L^q(d\mu_{\a})}
    \lesssim \big\|x^{a+\a+1/2-(2\a+1)/p}g\|_{L^p(d\mu_{\a})}
$$
for all $g\in L^p\big(x^{[a+\a+1/2-(2\a+1)/p]p}d\mu_{\a}\big)$. 
This combined with
Theorem \ref{thm:main} (ii) (with (e) replaced by (e')) gives the desired conclusion.
\end{proof}

Similarly to Theorem~\ref{thm:LpLq_bes}, Theorem~\ref{thm:LpLq_besN}
follows readily from the two lemmas, stated below, describing $L^p-L^q$ behavior
of two auxiliary operators with non-negative kernels which $\mathcal J^{\a,\s}$ splits into. More precisely,
we split the operator $\mathcal J^{\a,\s}$ according to the kernel splitting
$$
\mathcal H^{\a,\s}(x,y) = 
	\chi_{\{x \le 2, y \le 2\}} \mathcal H^{\a,\s}(x,y) + \chi_{\{x \vee y > 2\}}\mathcal H^{\a,\s}(x,y) 
    \equiv \mathcal H^{\a,\s}_{0}(x,y) + \mathcal H^{\a,\s}_{\infty}(x,y)
$$
and denote the resulting integral operators by $\mathcal J^{\a,\s}_{0}$ 
and $\mathcal J^{\a,\s}_{\infty}$, respectively.

\begin{lem}[{\cite[Lemma 4.4]{NoSt5} and \cite{NR}}] \label{lem:bes_loc_H}
Let $\a > -1$, $\s>0$ and $1 \le p,q \le \infty$.  
\begin{itemize}
\item[(a)] If $\a \ge -1/2$, then $\mathcal J^{\a,\s}_{0}$ 
is bounded from $L^p(dx)$ to $L^q(dx)$ if and only if
$$
\frac{1}p - 2\s\le \frac{1}q 
\quad \textrm{and} \quad	\bigg(\frac{1}p,\frac{1}q\bigg) \notin \big\{(2\s,0), (1,1-2\s)\big\}.
$$
\item[(b)] Let $\a < -1/2$. Then $L^p(dx)\subset \domain \mathcal J_{0}^{\a,\s}$ if and only if
$p > 2/(2\a+3)$. In this case
$\mathcal J^{\a,\s}_{0}$ is bounded from $L^p(dx)$ to $L^q(dx)$ if and only if
$$
\frac{1}p - 2\s \le \frac{1}q  \quad  \textrm{and} \quad \frac{1}q > -\a-\frac{1}2.
$$
\end{itemize}
\end{lem}
\begin{lem} \label{lem:bes_glob_H}
Let $\a > -1$, $\s>0$ and $1 \le p,q \le \infty$.  Then  $\mathcal J^{\a,\s}_{\infty}$ 
satisfies the positive $L^p-L^q$ mapping properties stated in 
Theorem \ref{thm:LpLq_besN} for $\mathcal J^{\a,\s}$, but is not 
bounded from $L^p(dx)$ to $L^q(dx)$ when $p > q$. 
\end{lem}

By Theorem \ref{thm:ker_bes} and \eqref{linkKH}, the kernel $\mathcal{H}_0^{\a,\s}(x,y)$ behaves precisely
in the same way as its analogue in the setting of Laguerre function expansions of Hermite type studied
in \cite[Section 4.2]{NoSt5}. Thus Lemma~\ref{lem:bes_loc_H} coincides with \cite[Lemma 4.4]{NoSt5}.
The latter is a consequence of the results in \cite{NR}, 
see the related comments in \cite[Section 4.2]{NoSt5}.
To prove Lemma \ref{lem:bes_glob_H} we will mostly appeal to the results obtained in the modified 
Hankel transform setting. Essentially, only the case $\a < -1/2$ requires new arguments.

\begin{proof}[{Proof of Lemma \ref{lem:bes_glob_H}}]
In view of \eqref{linkKH} and Theorem \ref{thm:ker_bes} (ii), for $\a \ge -1/2$ the kernel
$\mathcal H_{\infty}^{\a,\s}(x,y)$ is controlled by $H_{\infty}^{-1/2,\s}(x,y)$. 
Thus $\mathcal J_{\infty}^{\a,\s}$
inherits the $L^p-L^q$ boundedness of $J_{\infty}^{-1/2,\s}$ 
(note that $d\mu_{-1/2}$ is the Lebesgue measure).
This together with Lemma \ref{lem:bes_glob} gives the positive results of the lemma in case $\a \ge -1/2$.

Next observe that for any $\a > -1$, the two above mentioned kernels are comparable in the sense
of $\simeq \simeq$ if the arguments are,
see Theorem \ref{thm:ker_bes} (ii),
\begin{equation} \label{krel}
\mathcal H_{\infty}^{\a,\s}(x,y) \simeq \simeq H_{\infty}^{-1/2,\s}(x,y), \qquad x/2 < y < 2x.
\end{equation}
So to prove the required negative result we can use the counterexamples from (a)
of the proof of Lemma \ref{lem:bes_glob}, since they involve only comparable arguments of the kernel
(in connection with the case $p=\infty$, notice that 
$\mathcal{J}^{\a,\s}\boldsymbol{1}(x) \gtrsim 1$, $x > 4$, 
by \eqref{krel} and the proof of Lemma \ref{lem:ass}).

It remains to justify the $L^p-L^q$ boundedness in case $\a < -1/2$. Because of \eqref{krel} and
Lemma~\ref{lem:bes_glob} (or rather its proof that depends on $J^{\a,\s}_{\infty}$ only through
qualitatively sharp estimates of the kernel $H_{\infty}^{\a,\s}(x,y)$), 
it is enough to study the mutually dual integral operators
\begin{align*}
U_1^{\a,\s}f(x) & = \int_0^{x/2} \mathcal H_{\infty}^{\a,\s}(x,y)f(y)\, dy,\\
U_2^{\a,\s}f(x) & = \int_{2x}^{\infty} \mathcal H_{\infty}^{\a,\s}(x,y)f(y)\, dy.
\end{align*}
We will show that $U_1^{\a,\s}$ and $U_2^{\a,\s}$ are $L^p$-bounded for $p$ satisfying
$$
-\a-\frac{1}2 < \frac{1}p < \a + \frac{3}2.
$$
This will finish the proof, because of Corollary \ref{cor:LpLq_H} (ii) (recall that $\mathcal{K}^{\a,\s}(x,y)$
dominates $\mathcal{H}^{\a,\s}(x,y)$ pointwise) and an interpolation argument.

By \eqref{linkKH}, Theorem \ref{thm:ker_bes} (ii) and H\"older's inequality,
$$
|U_1^{\a,\s}f(x)| \lesssim x^{-\a-1/2} \exp\big(-cx\big) 
	\left\{ 
	\begin{array}{ll}
		x^{2\s-1}, & \s < 1/2 \\
		1+\log^+\frac{1}{x}, & \s=1/2 \\
		1, & \s > 1/2
	\end{array}
	\right\}
	 \big\|\chi_{\{y < x\}} y^{\a+1/2}\big\|_{p'} \|f\|_p.
$$
Since $p>\frac2{2\a+3}$, the $L^{p'}$ norm here is finite and comparable to $x^{\a+3/2-1/p}$. Then we get
$$
|U_1^{\a,\s}f(x)| \lesssim g(x) \|f\|_p, \qquad x >0,
$$
where
$$
g(x) = x^{1-1/p} \exp\big(-cx\big) 
	\begin{cases}
		x^{2\s-1}, & \s < 1/2, \\
		1+\log^+\frac{1}{x}, & \s=1/2, \\
		1, & \s > 1/2.
	\end{cases}
$$
Since $g \in L^p$, we see that $U_1^{\a,\s}$ is $L^p$-bounded.

Considering $U_{2}^{\a,\s}$, we recall that it is the dual of $U_1^{\a,\s}$ and use the already proved
result for $U_{1}^{\a,\s}$. 
\end{proof}

Finally, for the sake of completeness and, perhaps, reader's curiosity, we formulate an analogue of
Lemma \ref{lem:ass}. The proof is very similar to that of Lemma \ref{lem:ass}.
\begin{propo} 
Let $\a > -1$, $\s>0$ and $1 \le p \le \infty$. Then the estimates
$$
\|\mathcal H^{\a,\s}_{\infty}(x,\cdot)\|_p\simeq
\begin{cases}
        x^{\a+1/2}, & x\le1,\\
        1, & x>1,
    \end{cases}
$$
hold provided that $p$ satisfies $\frac{1}p > 1-2\s$ and, in addition, $\frac{1}p > - \a - 1/2$ in case
$\a < -1/2$. 

Moreover, for the remaining $p$ we have
\begin{equation*}  
\|\mathcal H^{\a,\s}_{ \infty}(x,\cdot)\|_p=\infty, \qquad x>4.
\end{equation*}
\end{propo}

\subsection{$\boldsymbol{L^p-L^q}$ estimates in the setting of the Hankel-Dunkl transform}

We will argue along the lines of the proof of \cite[Theorem 2.6]{NoSt5}. 
The following notation will be useful. For a function $f$ on $\mathbb{R}$, define $f_+$ and $f_{-}$
as functions on $\mathbb{R}_+$ given by $f_{\pm}(x) = f(\pm x)$, $x>0$. In a similar way, let
$\mathbb{K}^{\a,\s}_{+}$ and $\mathbb{K}^{\a,\s}_{-}$ be the kernels on $\mathbb{R}_+\times \mathbb{R}_+$
determined by $\mathbb{K}_{\pm}^{\a,\s}(x,y)=\mathbb{K}^{\a,\s}(x,\pm y)$, $x,y > 0$. 
Denote the corresponding integral operators related to the measure space $(\mathbb{R}_+,d\mu_{\a})$ by
$\mathbb{I}_{\pm}^{\a,\s}$, respectively.

It is clear that for any fixed $1\le p \le \infty$,
$$
\|f\|_{L^p(dw_{\a})} \simeq \|f_+\|_{L^p(d\mu_{\a})} + \|f_{-}\|_{L^p(d\mu_{\a})}.
$$
Further, by the symmetry of the kernel, $\mathbb{K}^{\a,\s}(-x,y)=\mathbb{K}^{\a,\s}(x,-y)$, 
and the symmetry of $w_{\a}$,
\begin{equation} \label{dec2}
\big(\mathbb I^{\a,\s}f\big)_\pm=\mathbb{I}^{\a,\s}_+(f_\pm)+ \mathbb{I}^{\a,\s}_-(f_\mp).
\end{equation}

\begin{proof}[{Proof of Theorem \ref{thm:main_D}}]
We first prove the sufficiency in (i). Take $a$ satisfying the relevant condition and assume that
$f \in L^p(|x|^{ap}dw_{\a})$. 
Then 
$f_{\pm} \in L^p(x^{ap}d\mu_{\a})$. 
Since, in view of \eqref{kkp1},
the kernels $\mathbb{K}_{\pm}^{\a,\s}(x,y)$ are controlled by ${K}^{\a,\s}(x,y)$, from 
Theorem \ref{thm:main} (i) it follows that $f_+,f_- \in \domain \mathbb{I}^{\a,\s}_{\pm}$.
Now \eqref{dec2} gives the desired conclusion.

To show the necessity in (i), consider $f$ such that $f_-\equiv 0$. Then 
$(\mathbb{I}^{\a,\s}f)_+ = \mathbb{I}^{\a,\s}_+(f_+)$. We see that if $f \in \domain \mathbb{I}^{\a,\s}$, then
$f_+ \in \domain \mathbb{I}^{\a,\s}_+$ and so $f_+ \in \domain I^{\a,\s}$, because of \eqref{kkp2}.
This means that if 
$L^p(|x|^{ap}dw_{\a}) \subset \domain \mathbb{I}^{\a,\s}$,
then $L^p(x^{ap}d\mu_{\a}) \subset \domain I^{\a,\s}$,
and so the condition in question must be satisfied in virtue of Theorem \ref{thm:main} (i).

Proving (ii), assume first that (a)-(e) are satisfied. Observe that,
because of \eqref{kkp1} and Theorem~\ref{thm:main} (ii), the operators $\mathbb{I}^{\a,\s}_{\pm}$ satisfy
\begin{equation} \label{in6}
\|x^{-b}\, \mathbb{I}^{\a,\s}_{\pm} g\|_{L^q(d\mu_{\a})} \lesssim \|x^a g\|_{L^p(d\mu_{\a})}
\end{equation}
uniformly in $g \in L^p(x^{ap}d\mu_{\a})$. 
Therefore, with \eqref{dec2} in mind, we can write
\begin{align*}
\| |x|^{-b}\, \mathbb{I}^{\a,\s}f\|_{L^q(dw_{\a})} & \simeq
\| x^{-b} (\mathbb{I}^{\a,\s}f)_+\|_{L^q(d\mu_{\a})} 
	+ \| x^{-b} (\mathbb{I}^{\a,\s}f)_-\|_{L^q(d\mu_{\a})} \\
& \le \| x^{-b}\, \mathbb{I}_+^{\a,\s}(f_+)\|_{L^q(d\mu_{\a})}
	+ \| x^{-b}\, \mathbb{I}_-^{\a,\s}(f_-)\|_{L^q(d\mu_{\a})} \\
& \quad 	+ \| x^{-b}\, \mathbb{I}_+^{\a,\s}(f_-)\|_{L^q(d\mu_{\a})}
	+ \| x^{-b}\, \mathbb{I}_-^{\a,\s}(f_+)\|_{L^q(d\mu_{\a})} \\
& \lesssim \|x^a f_+\|_{L^p(d\mu_{\a})} + \|x^a f_-\|_{L^p(d\mu_{\a})} \\
& \simeq \| |x|^a f\|_{L^p(dw_{\a})}.
\end{align*}

Finally, to show the necessity part in (ii) we consider $f$ such that $f_- \equiv 0$. Then, in view of
\eqref{kkp2}, $\mathbb{I}_+^{\a,\s}(f_+) \simeq I^{\a,\s}(f_+)$. Therefore, by Theorem \ref{thm:main} (ii),
conditions (a)-(e) are necessary for $\mathbb{I}^{\a,\s}_+$ to satisfy \eqref{in6}. But such an estimate
is implied by the analogous one for $\mathbb{I}^{\a,\s}$, because
$\|x^{-b}\,\mathbb{I}_+^{\a,\s}(f_+)\|_{L^q(d\mu_{\a})} = \|x^{-b}(\mathbb{I}^{\a,\s}f)_+\|_{L^q(d\mu_{\a})}
\le \||x|^{-b}\,\mathbb{I}^{\a,\s}f\|_{L^q(dw_{\a})}$ and 
$\||x|^a f\|_{L^p(dw_{\a})} = \|x^a f_+\|_{L^p(d\mu_{\a})}$. 
The conclusion follows.
\end{proof}

\begin{proof}[{Proof of Theorem \ref{thm:LpLq_D}}] 
Items (i) and (ii) are special cases of the corresponding parts in Theorem~\ref{thm:main_D}. 
To show (iii) we use Theorem \ref{thm:LpLq} (iii) and arguments similar to those from the proof of 
Theorem \ref{thm:main_D}. Details are left to the reader.
\end{proof}

\begin{proof}[{Proof of Theorem \ref{thm:LpLq_besD}}]
The reasoning relies on arguments analogous to those from the proof of Theorem \ref{thm:main_D}, combined
with \eqref{Bkkp1}, \eqref{Bkkp2}, Theorem  \ref{thm:LpLq_bes} and an analogue of \eqref{dec2}. 
We omit the details.
\end{proof}


\end{document}